\documentclass[oneside,english]{amsart}
\usepackage[T1]{fontenc}
\usepackage[latin9]{inputenc}
\usepackage{color}
\usepackage{babel}
\usepackage{units}
\usepackage{amsthm}
\usepackage{amsbsy}
\usepackage{amssymb}
\usepackage{esint}
\usepackage[unicode=true,pdfusetitle,
 bookmarks=true,bookmarksnumbered=false,bookmarksopen=false,
 breaklinks=false,pdfborder={0 0 0},backref=false,colorlinks=true]
 {hyperref}

\makeatletter

\newcommand{\noun}[1]{\textsc{#1}}

\numberwithin{equation}{section}
\numberwithin{figure}{section}
\theoremstyle{plain}
\newtheorem{thm}{\protect\theoremname}[section]
  \theoremstyle{definition}
  \newtheorem{defn}[thm]{\protect\definitionname}
  \theoremstyle{definition}
  \newtheorem*{example*}{\protect\examplename}
  \theoremstyle{plain}
  \newtheorem{lem}[thm]{\protect\lemmaname}
  \theoremstyle{remark}
  \newtheorem*{rem*}{\protect\remarkname}
 \theoremstyle{definition}
 \newtheorem*{defn*}{\protect\definitionname}
  \theoremstyle{plain}
  \newtheorem{cor}[thm]{\protect\corollaryname}
  \theoremstyle{plain}
  \newtheorem{prop}[thm]{\protect\propositionname}
  \theoremstyle{plain}
  \newtheorem{question}[thm]{\protect\questionname}
  \theoremstyle{plain}
  \newtheorem*{question*}{\protect\questionname}

\def
 \nicerfrac#1#2{ 
     \raise.5ex\hbox{$#1$}%
     \kern-.1em\bigg / \kern-.15em%
     \lower.25ex\hbox{$#2$
  }
}

\makeatother

  \providecommand{\corollaryname}{Corollary}
  \providecommand{\definitionname}{Definition}
  \providecommand{\examplename}{Example}
  \providecommand{\lemmaname}{Lemma}
  \providecommand{\propositionname}{Proposition}
  \providecommand{\questionname}{Question}
  \providecommand{\remarkname}{Remark}
\providecommand{\theoremname}{Theorem}

\begin{document}

\title{On Cheeger and Sobolev differentials in metric measure spaces}

\author{Martin Kell}

\address{Mathematisches Institut, Universität Tübingen, Tübingen, Germany}

\email{martin.kell@math.uni-tuebingen.de}
\begin{abstract}
Recently Gigli developed a Sobolev calculus on non-smooth spaces using
module theory. In this paper it is shown that his theory fits nicely
into the theory of differentiability spaces initiated by Cheeger,
Keith and others. A relaxation procedure for $L^{p}$-valued subadditive
functionals is presented and a relationship between the module generated
by a functional and the one generated by its relaxation is given.
In the framework of differentiability spaces, which includes so called
PI- and $RCD(K,N)$-spaces, the Lipschitz module is pointwise finite
dimensional. A general renorming theorem together with the characterization
above shows that the Sobolev spaces of such spaces are reflexive.
\end{abstract}
\maketitle
\global\long\def\Pfm{\mathsf{Pfm}}
\global\long\def\esssup{\operatorname{ess}\operatorname{sup}}
\global\long\def\supp{\operatorname{supp}}
\global\long\def\lip{\operatorname{lip}}
\global\long\def\Lip{\operatorname{Lip}}
\global\long\def\Lip{\operatorname{Lip}}
\global\long\def\Norm{\mathsf{Norm}}
\global\long\def\Scalar{\mathsf{Scalar}}

In the recent years the analysis on metric measure spaces has made
a lot of progress. Initiated by \cite{Hajasz1996,Heinonen1996} (see
also \cite{Heinonen1998,Koskela1997}) a relaxed notion of gradients,
more precisely the norm of a relaxed gradient, was defined. Together
with abstract Poincar\'e and doubling conditions a theory of minimizers
resembling harmonic functions was developed (see \cite{Bjorn2011,Heinonen2015}).
Spaces satisfying those conditions are now called PI-spaces. Those
ideas were also incorporated in Cheeger's generalized Rademacher theorem
\cite{Cheeger1999} where Cheeger proved existence of ``chart functions''
$((x_{j}:A_{i,n}\to\mathbb{R}^{n})_{j=1}^{n}$ where $\{A_{i,n}\}$
is a Borel partition such that for any Lipschitz function $f$ there
are generalized Lipschitz differentials $Df$ with $Df_{|A_{i,n}}:A_{i,n}\to\mathbb{R}^{n}$
and for almost all $z_{0}\in A_{i,n}$ 
\[
f(z)=f(z_{0})+\sum_{j=1}^{n}(Df(z))_{j}\cdot(x_{j}(z)-x_{j}(z_{0}))+o(d(z,z_{0})).
\]
Later on, Keith \cite{Keith2004} noticed that a weaker condition,
called $\Lip$-$\lip$-condition, is sufficient to obtain the same
structure. Bate \cite{Bate2014} developed a dual formalism of this
by constructing ``sufficiently many tangent curves'', which give
a more precise characterization of the differentials. 

On PI-spaces Cheeger also showed that his differentials could be used
to obtain a Dirichlet form which allow for a natural definition of
linear Laplace operator and thus PDE-like harmonic functions. Combined
with a form of the Bakry-\'Emery condition those differentials were
used to show Lipschitz continuity of harmonic functions \cite{Koskela2003}.
The same structure was also used in the $L^{p}$-theory to generalize
result from the smooth setting to the metric setting.

Independently there was a need for a PDE-like theory in metric spaces.
Using the theory of optimal transport it could be shown that lower
bounds on the Ricci curvature are equivalent to convexities of entropy
functionals in the space of probability measures equipped with a transportation
distance \cite{VonRenesse2005}. As it turns out the gradient flow
of the entropy can be identified with the natural heat flow induced
by the gradient structure \cite{JKO1998}. Ambrosio-Gigli-Savar\'e
\cite{Ambrosio2013} developed for metric spaces a sufficiently strong
calculus via relaxed gradients to give such an identification. This
was used in \cite{AGS2011,Erbar2013} to show that on a subclass behaving
like generalized Riemannian manifolds, the lower Ricci bound in terms
of optimal transport are equivalent to the lower bounds based on the
analytic condition of condition Bakry-\'Emery. In a different paper
\cite{Ambrosio2011} Ambrosio-Gigli-Savar\'e showed that for not
too bad metric spaces, their relaxed notion of gradient agrees with
the weak upper gradient of Heinonen-Koskela-MacManus. 

In a recent work \cite{Gigli2014}, Gigli developed a Sobolev calculus
which resembles the one in the smooth setting. He first constructs
$L^{p}$-integrable ``$1$-forms'' and assigns to each Lipschitz
function a unique differential whose norm is given by the relaxed
slope/gradient. Closedness of this assignment shows that any Sobolev
function a unique (Sobolev) differential. Using those ideas he develops
a (weak) second order calculus on Riemannian-like spaces with lower
Ricci bounds.

In this paper, a precise description of the Sobolev differentials
in terms of the Lipschitz differentials is presented. The main ingredient
of Gigli's construction was the functional 
\[
f\mapsto\int|Df|_{*}^{p}d\mathbf{m}
\]
which is a relaxation of 
\[
f\mapsto\int(\lip f)^{p}d\mathbf{m}.
\]
We show that one may construct a Lipschitz module using the latter
and regard the cotangent module as a relaxed version of the Lipschitz
module. With the help of a general relaxation procedure is presented
it can be shown that the cotangent module is a quotient space of the
Lipschitz module and a submodule called the $\mbox{Lipschitz}_{0}$
module. This characterization immediately shows that on Lipschitz
differentiable spaces the cotangent modules are locally finitely generated
so that the Sobolev spaces are (super)reflexive. In PI-spaces which
are infinitesimally Hilbertian one can even show that the Cheeger
differential structure is just a certain representation (w.r.t. some
basis) of the Sobolev differentials. 

This characterization can be explained as follows: the linear operator
assigning to each Lipschitz map its Lipschitz differential is in general
not closable. Its closure assigns to each Sobolev map a whole affine
subspace of ``Lipschitz $1$-forms''. The relaxed slope is nothing
but the distance of that affine subspace from the origin which shows
that the norm of the Sobolev differentials is a quotient norm. Pointwise
minimality of the relaxed slope shows that the same holds for the
generated module. In principle, one obtains the Sobolev spaces characterization
more directly by looking at the Pseudo-Sobolev space generated by
Lipschitz functions and their Lipschitz differential. We present the
module version as it translates directly into the language of Lipschitz
differentiable structure.

The paper is structured as follows: In the first section normed modules
are introduced and a representation theorem of locally finite dimensional
modules is given. Afterwards it is shown how to obtain an $L^{p}$-normed
module given an $L^{p}$-valued functional that behaves like a pointwise
norm. The next section gives a general relaxation procedure and it
is shown that a module and its relaxation can be characterized precisely
if the module is weakly reflexive. The third section applies the abstract
theory to Lipschitz and Sobolev functions. It is shown that on differentiability
spaces the Lipschitz module is locally finite dimensional which shows
that all Sobolev space $W^{1,p}(M,\mathbf{m})$ are reflexive. Then
we present the rigidity theorem of Lipschitz differential measures
in $\mathbb{R}^{n}$ and some of its conclusion. In the end the relationship
of Cheeger and Sobolev differentials is shown. The appendix contains
an account on norms and a (unique) choice of scalar product.

Throughout the paper we have the following assumption: $(M,d,\mathbf{m})$
is a complete metric measure space with $\mathbf{m}$ a Radon measure.

\section{Normed Modules}

In this section the theory of $L^{p}(\mathbf{m})$-normed spaces is
introduced. In the context of metric spaces $L^{\infty}(\mathbf{m})$-normed
premodules appeared first in the work of Weaver \cite{Weaver2000}.
Later Gigli \cite{Gigli2014} defined more general $L^{p}(\mathbf{m})$-normed
modules to define generalized $1$-forms and Sobolev differentials.
We present Gigli's construction independently of metric spaces and
give a more precise representation of locally finite dimensional modules.
That representation can be seen as an abstract version of Cheeger
renorming theorem yielding the Cheeger differentiable structure.
\begin{defn}
[$L^p(\mathbf{m})$-normed module] A Banach space $(\mathcal{M},\|\cdot\|_{\mathcal{M}})$
is \emph{$L^{p}(\mathbf{m})$-normed premodule} for $p\in[1,\infty]$
if there is a bilinear map $L^{\infty}(\mathbf{m})\times\mathcal{M}\to\mathcal{M}$,
$(f,v)\mapsto f\cdot v$ and a map $|\cdot|:\mathcal{M}\to L^{p}(\mathbf{m})$
such that for every $v\in\mathcal{M}$ and $f,g\in L^{\infty}(\mathbf{m})$
\begin{eqnarray*}
(fg)\cdot v & = & f\cdot(g\cdot v)\\
\mathbf{1}\cdot v & = & v\\
\|v\|_{\mathcal{M}} & = & \||v|\|_{L^{p}(\mathbf{m})}\\
|f\cdot v| & = & |f||v|\quad\mbox{\ensuremath{\mathbf{m}}}\mbox{-a.e.},
\end{eqnarray*}
where $\mathbf{1}$ is the $L^{\infty}$-function which is $1$ everywhere.

The premodule is called a \emph{$L^{p}(\mathbf{m})$-normed module}
if the additional two properties hold:
\begin{itemize}
\item (\noun{Locality}) Assume for $v\in\mathcal{M}$ and $\{A_{n}\}_{n\in\mathbb{N}}$
it holds $\chi_{A_{n}}\cdot v=0$ then $\chi_{\cup_{n}A_{n}}\cdot v=0$.
\item (\noun{Glueing}) For every sequence $(v_{n})_{n\in\mathbb{N}}$ in
$\mathcal{M}$ and sequence of Borel sets $(A_{n})_{n\in\mathbb{N}}$
such that
\begin{eqnarray*}
\chi_{A_{i}\cap A_{j}}\cdot v_{i}=\chi_{A_{i}\cap A_{j}}\cdot v_{j} & \hspace{1em}\mbox{and}\hspace{1em} & \limsup_{n\to\infty}\|\sum_{i=1}^{n}\chi_{A_{i}}\cdot v_{i}\|_{\mathcal{M}}<\infty
\end{eqnarray*}
 there is a $v\in\mathcal{M}$ such that
\begin{eqnarray*}
\chi_{A_{i}}\cdot v=\chi_{A_{i}}\cdot v_{i} & \hspace{1em}\mbox{and}\hspace{1em} & \|v\|_{\mathcal{M}}\le\liminf_{n\to\infty}\|\sum_{i=1}^{n}\chi_{A_{i}}\cdot v_{i}\|_{\mathcal{M}}.
\end{eqnarray*}

\end{itemize}
\end{defn}
A closed subspace $\mathcal{N}$ of $\mathcal{M}$ is said to be a
submodule if it is also an $L^{p}(\mathbf{m})$-normed module. In
particular, it needs to be stable under $L^{\infty}(\mathbf{m})$-multiplication
and closed w.r.t. the locality and glueing constructions. In case
$\mathcal{N}$ is a submodule we can equip the quotient space $\mathcal{M}_{\mathcal{N}}=\mathcal{M}/\mathcal{N}$
with the following norm 
\[
\|[v]\|_{\mathcal{M}/\mathcal{N}}=\inf_{w\in\mathcal{N}}\|v+w\|_{\mathcal{M}}.
\]
Then $\mathcal{M}_{\mathcal{N}}$ has natural $L^{p}(\mathbf{m})$-normed
module structure satisfying the locality and glueing principle (see
\cite[Proposition 1.2.14]{Gigli2014}). 
\begin{example*}
(\noun{vector-valued spaces}) The space $L^{p}(M,\mathbf{m},\mathbb{R}^{n},|\cdot|)$
of vector-valued $L^{p}$-functions such that
\[
\|v\|_{L^{p}}^{p}=\int|v(x)|_{x}^{p}d\mathbf{m}
\]
where is a measurable map $x\mapsto|\cdot|_{x}$ into the space of
norms $\Norm(\mathbb{R}^{n})$ defined on $\mathbb{R}^{n}$ (see appendix
for properties of $\Norm(\mathbb{R}^{n})$). In case $n=0$ the spaces
is just the trivial vector space. 

(\noun{$A$-submodule}) Let $\mathcal{M}$ be an $L^{p}(\mathbf{m})$-normed
module. For each measurable $A\subset M$ the $A$-submodule $\mathcal{M}_{A}$
is the submodule of all $v\in\mathcal{M}$ with $\chi_{A}\cdot v=v$,
i.e. $\{v\ne0\}\subset A$ (compare \cite[Proposition 1.4.6]{Gigli2014}).
It is easy to see that this is also a natural $L^{p}(\mathbf{m}_{|A})$-normed
module on $A$ and that $v\mapsto\chi_{A}v$ is a distance non-increasing
projection on the submodule. In case of vector-valued spaces this
means one may unambiguously write $L^{p}(A,\mathbf{m}_{|A},\mathbb{R}^{n},|\cdot|^{'})\le L^{p}(M,\mathbf{m},\mathbb{R}^{n},|\cdot|)$
if $|\cdot|_{x}^{'}=|\cdot|_{x}$ for $\mathbf{m}$-almost all $x\in A$.
Note that automatically $L^{p}(A,\mathbf{m},\mathbb{R}^{n},|\cdot|^{'})=\{0\}$
if $\mathbf{m}(A)=0$. 

(\noun{$L^{p}$-products}) Assume $\mathcal{M}_{1}$ and $\mathcal{M}_{2}$
are two $L^{p}(\mathbf{m})$-normed modules. Then the $L^{p}$-product
$\mathcal{M}$ of the modules is an $L^{p}$-normed module, i.e. for
$v=v_{1}+v_{2}\in\mathcal{M}=\mathcal{M}_{1}\oplus\mathcal{M}_{2}$
we set 
\[
|v|^{p}=|v_{1}|^{p}+|v_{2}|^{p}
\]
for $p\in[1,\infty)$ and $|v|=\max\{|v_{1}|,|v_{2}|\}$ if $p=\infty$.

One may take the $L^{p}$-product of countable many modules by requiring
that $v=\sum_{i}v_{i}\in\mathcal{M}$ iff $|v|$ is well-defined,
i.e. $(|v_{i}|)\in\ell^{p}$ $\mathbf{m}$-almost everywhere, and
$|v|$ is in $L^{p}(\mathbf{m})$. \end{example*}
\begin{lem}
[Module parition] Assume $\{A_{n}\}_{n\in\mathbb{N}}$ is a Borel
partition of $M$. Then any $L^{p}$-normed module $\mathcal{M}$
is the $L^{p}$-product of the $A_{n}$-submodules.\end{lem}
\begin{proof}
If $v\in\mathcal{M}$ then $v_{n}=\chi_{A_{n}}v\in\mathcal{M}_{A_{n}}$.
Furthermore, it holds 
\begin{eqnarray*}
\|v\|_{\mathcal{M}}^{p} & = & \sum_{n\in\mathbb{N}}\int_{A_{n}}|v|^{p}d\mathbf{m}\\
 & = & \sum_{n\in\mathbb{N}}\|v_{n}\|_{\mathcal{M}_{A_{n}}}^{p}.
\end{eqnarray*}
In particular, $x\mapsto(|v_{n}|_{x})_{n\in\mathbb{N}}\in\ell^{p}$
for $\mathbf{m}$-almost everywhere.\end{proof}
\begin{lem}
If $p\in(1,\infty)$ then any vector-valued $L^{p}$-space and any
$L^{p}$-product of reflexive modules is reflexive.\end{lem}
\begin{rem*}
Actually the proof shows that any vector-valued $L^{p}$-space is
super-reflexive. Furthermore, the $L^{p}$-product of finitely many
modules is also super-reflexive if each factor was. \end{rem*}
\begin{proof}
The fact that $L^{p}(M,\mathbb{R}^{n},|\cdot|)$ is (super)reflexive
follows from Theorem \ref{thm:John-scalar}. Indeed, if $\Phi$ denotes
the John scalar product selector, then $x\mapsto|\cdot|^{'}=\Phi(|\cdot|)^{\frac{1}{2}}$
is also measurable and $|\cdot|^{'}\le|\cdot|\le n|\cdot|^{'}$ for
$\mathbf{m}$-almost all $x\in M$. Thus $L^{p}(M,\mathbb{R}^{n},|\cdot|')$
is also an $L^{p}$-normed module and one can show that it is $p$-uniformly
smooth if $p\in(1,2]$ and $p$-uniformly convex if $p\in[2,\infty)$.
In particular, it is (super-)reflexive.

The second fact is well-known. Just note that the dual of an $L^{p}$-product
is the $L^{q}$-product of the duals. 
\end{proof}
In case of vector-valued $L^{p}$-spaces there is finite set of generators
(pointwise) given by the standard basis times a positive $L^{\infty}\cap L^{p}$-integrable
function. 

We say that $\{v_{1},\ldots,v_{n}\}\subset\mathcal{M}$ is \emph{locally
independent} on $A$ if for all $x\in A$ the map $v\mapsto|v|_{x}$
is a norm for the space spanned $\{v_{1},\ldots,v_{n}\}$. For $L^{p}$-normed
modules this definition agrees with more general one defined in \cite[Definition 1.4.1]{Gigli2014}
for general $L^{\infty}(\mathbf{m})$-modules. 

Given that terminology one may wonder what the maximal number of locally
independent sets is and what happens if it is bounded globally. As
it turns out the vector-valued $L^{p}$-spaces are the building blocks
of those modules.
\begin{defn}
[Local finite dimensional]A module $\mathcal{M}$ is \emph{locally
finite dimensional} if there is a partition $\{E_{n}\}_{n=0}^{\infty}\cup\{E_{\infty}\}$
with $\mathbf{m}(E_{\infty})=0$ such that every maximal independent
set $\{v_{i}\}$ on subsets $E$ of $E_{n}$ with $\mathbf{m}(\{v_{i}\ne0\}\backslash E)=0$
has cardinality $n$. For every $x\in E_{n}$ define the local dimension
at $x$ by $\dim(\mathcal{M},x)=n$. The module is said to have \emph{local
finite dimension bounded }$N$ if $\mathbf{m}(E_{n})=0$ for $n>N$,
i.e. $\dim(\mathcal{M},x)\le n$ for $\mathbf{m}$-almost all $x\in M$.
\end{defn}
By \cite[Proposition 1.4.5]{Gigli2014} the decomposition $\{E_{n}\}_{n=0}^{\infty}\cup\{E_{\infty}\}$
always exists with possibly $\mathbf{m}(E_{\infty})>0$. 
\begin{thm}
[Representation Theorem] \label{thm:reflexive-renorm}Every locally
finite dimensional $L^{p}$-normed module $\mathcal{M}$ is isometric
to an infinite $L^{p}$-product of vector-valued $L^{p}$-spaces,
i.s. 
\[
\mathcal{M}\cong\oplus_{n=0}^{\infty}L^{p}(E_{n},\mathbb{R}^{n},|\cdot|).
\]
If the dimension is bounded by $N$ then $\mathcal{M}$ admits a uniformly
equivalent norm such that $x\mapsto|\cdot|_{x}^{'}$ is almost everywhere
induced by a scalar product. In particular, $\mathcal{M}$ is (super-)reflexive
if $p\in(1,\infty)$.\end{thm}
\begin{rem*}
(1) Reflexivity of any locally finite dimensional $L^{p}$-normed
module was already shown by Gigli \cite[Theorem 1.4.7]{Gigli2014}
without the vector-valued representation. 

(2) In an abstract language this is exactly what Cheeger \cite{Cheeger1999}
did by constructing his Cheeger differential structure out of the
Lipschitz differential structure. Later we will be more precise about
that construction.\end{rem*}
\begin{proof}
By the lemmas above it suffices to show that $\mathcal{M}_{E_{n}}\cong L^{p}(E_{n},\mathbb{R}^{n},|\cdot|)$. 

For $n=0$ there is nothing to prove. Assume $n>0$ and choose a (module)
basis $\{v_{1},\ldots,v_{n}\}$ for $\mathcal{M}_{E_{n}}$ such that
the span is closed in $\mathcal{M}$ (see \cite[Proposition 1.4.6]{Gigli2014}).
If $v\in\mathcal{M}_{E_{n}}$ then there are $L^{\infty}$-maps $\alpha_{i}$
such that $v=\sum_{i=1}^{n}\alpha_{i}v_{i}$ and $|v|_{x}=F_{x}(\boldsymbol{\alpha})$
for a unique ($\mathbf{m}$-a.e.) norm $F_{x}$ on $\mathbb{R}^{n}$.
This proves that $v\mapsto\boldsymbol{\alpha}$ is an isomorphism
from $\mathcal{M}_{E_{n}}$ to $L^{p}(E_{n},\mathbb{R}^{n},F)$.

For the last part note that if $\mathbf{m}(E_{n})=0$ for $n\ge N$
then $L^{p}(E_{n},\mathbb{R}^{n},F)=0$ for $n\ge N$ so that $|\cdot|^{'}\le F_{x}\le N|\cdot|^{'}$
for $\mathbf{m}$-almost all $x\in M$ where $|\cdot|^{'}$ is the
norm which is induced by the John scalar product if $F_{x}$, see
proof of the lemma above and Theorem \ref{thm:John-scalar}. 
\end{proof}

\subsection*{A general construction.}

In order to simplify the discussion on the relationship of the $L^{p}$-Lipschitz-
and $L^{p}$-cotangent modules we present an abstract version of Gigli's
cotangent module construction (see \cite[Section 2.2.1]{Gigli2014}).

Assume $p\in[1,\infty)$ and denote by $L_{+}^{p}(\mathbf{m})$ the
set of non-negative $L^{p}$-integrable functions. Let $L^{0}(M)$
be the space of measurable functions. There is a function $\mathcal{F}_{p}:D(\mathcal{F}_{p})\to L_{+}^{p}(\mathbf{m}),f\mapsto g_{f}$
with $D(\mathcal{F}_{p})\subset L^{0}(\mathbf{m})$ such that for
all $f,h\in D(\mathcal{F}_{p})$ and $\lambda\in\mathbb{R}$ it holds 

\begin{eqnarray*}
g_{f+h} & \le & g_{f}+g_{h}\\
g_{\lambda f} & = & |\lambda|g_{f}.
\end{eqnarray*}

\begin{rem*}
A similar construction also works for $p=\infty$ the evaluation is
then w.r.t. the $L^{\infty}$-norm. $\Pfm$ (see below) is then defined
whenever $\sup_{i}\{\|g_{f_{i}}\|\}<\infty$.
\end{rem*}
Define the set $\Pfm$ generated by $\mathcal{F}_{p}$ as follows
\[
\begin{array}{ccll}
\Pfm & := & \bigg\{\{(f_{i},A_{i})\}_{i\in\mathbb{N}}\,|\, & (A_{i})_{i\in\mathbb{N}}\mbox{ is a Borel partition of \ensuremath{M}}\\
 &  &  & f_{i}\in D(\mathcal{F}_{p}),\mbox{ and }\sum_{i\in\mathbb{N}}\int_{A_{i}}g_{f_{i}}^{p}dx<\infty\bigg\}.
\end{array}
\]

On $\Pfm$ define the equivalence relation $\{(f_{i},A_{i})\}_{i\in\mathbb{N}}\sim\{(h_{j},B_{j})\}_{j\in\mathbb{N}}$
if 
\[
g_{f_{i}-h_{j}}=0\quad\mathbf{m}\mbox{-almost every on \ensuremath{A_{i}\cap B_{j}}.}
\]
Indeed, reflexivity follows from homogeneity, symmetry follows from
the definition, and transitivity from subadditivity. 

Now it is easy to verify that $\Pfm/\sim$ is a vector space such
that 
\[
[(f_{i},A_{i})_{i\in\mathbb{N}}]+[(h_{j},B_{j})_{j\in\mathbb{N}}]=[(f_{i}+h_{j},A_{i}\cap B_{j})_{i,j\in\mathbb{N}}]
\]
and 
\[
[(\lambda f_{i},A_{i})_{i\in\mathbb{N}}]=\lambda[(f_{i},A_{i})_{i\in\mathbb{N}}].
\]

\begin{rem*}
Below we deal with multiple $L^{p}$-valued functional. In that case
we may put an $\mathcal{F}_{p}$-index at equivalence relation $[\cdot]$
to avoid confusion. 
\end{rem*}
If $a\in L^{\infty}(\mathbf{m})$ is a simple function, i.e. $a=\sum_{j\in\mathbb{N}}a_{j}\chi_{B_{j}}$
with $|a_{j}|\le\|a\|_{\infty}$ and $(B_{i})_{i}$ a Borel partition
of $M$ then 
\[
a[(f_{i},A_{i})_{i\in\mathbb{N}}]=[(a_{j}f_{i},A_{i}\cap B_{j})_{i,j\in\mathbb{N}}].
\]
Also define for each $v=[(f_{i},A_{i})_{i\in\mathbb{N}}]\in\Pfm/\sim$
a measurable map $x\mapsto|v|_{x}$ by 
\[
|[(f_{i},A_{i})_{i\in\mathbb{N}}]|=g_{f_{i}}\quad\mathbf{m}\mbox{-almost every on \ensuremath{A_{i}}},\forall i\in\mathbb{N}.
\]
The assignment $|\cdot|$ is a (pointwise) semi-norm on $\Pfm/\sim$
compatible with multiplication by simple $L^{\infty}(\mathbf{m})$-functions:
\begin{eqnarray*}
|[(f_{i}+h_{j},A_{i}\cap B_{j})_{i,j\in\mathbb{N}}]| & \le & |[(f_{i},A_{i})_{i\in\mathbb{N}}]|+|[(h_{j},B_{j})_{j\in\mathbb{N}}]|\\
|\lambda[(f_{i},A_{i})_{i\in\mathbb{N}}]| & = & |\lambda||[(f_{i},A_{i})_{i\in\mathbb{N}}]|\\
|a[(f_{i},A_{i})_{i\in\mathbb{N}}]| & = & |h||[(f_{i},A_{i})_{i\in\mathbb{N}}]|.
\end{eqnarray*}
 Finally, it is readily verified that $\|\cdot\|_{\mathcal{F}_{p}}:\Pfm/\sim\to[0,\infty)$
defined by 
\[
\|[(f_{i},A_{i})_{i\in\mathbb{N}}]\|_{\mathcal{F}_{p}}=\||[(f_{i},A_{i})_{i\in\mathbb{N}}]|\|_{L^{p}}=\sum_{i\in\mathbb{N}}\int_{A_{i}}g_{f_{i}}^{p}d\mathbf{m}
\]
is a (possibly incomplete) norm on $\Pfm/\sim$. From the properties
above one can show that the completion of $\Pfm$ will be an $L^{p}$-normed
module.
\begin{defn*}
[$\mathcal{F}_p$-module] The $L^{p}$-normed module $(\mathcal{M}_{\mathcal{F}_{p}},\|\cdot\|_{\mathcal{F}_{p}})$
is the completion of $(\Pfm/\sim,\|\cdot\|_{\mathcal{F}_{p}})$.
\end{defn*}
As the construction is essentially unique we just say $(\mathcal{M}_{\mathcal{F}_{p}},\|\cdot\|_{\mathcal{F}_{p}})$
is the $\mathcal{F}_{p}$-module. Also define the operator $d_{\mathcal{F}_{p}}:D(\mathcal{F}_{p})\to\mathcal{M}_{\mathcal{F}_{p}}$
by 
\[
d_{\mathcal{F}_{p}}(f)=[f,M].
\]

\begin{lem}
\label{lem:subgenerator}Assume $\mathcal{F}_{p}^{'}$ and $\mathcal{F}_{p}$
are two functionals such that $D(\mathcal{F}_{p}^{'})\subset D(\mathcal{F}_{p})$
such that $\mathcal{F}_{p}^{'}(f)=\mathcal{F}_{p}(f)$ for all $f\in D(\mathcal{F}_{p}^{'})$.
Then $\mathcal{M}_{\mathcal{F}_{p}^{'}}$ can be uniquely identified
with a submodule of $\mathcal{M}_{\mathcal{F}_{p}}$. \end{lem}
\begin{cor}
Under the above assumptions. If for every $f\in D(\mathcal{F}_{p})$
there is a sequence $(f_{n})_{n\in\mathbb{N}}\subset D(\mathcal{F}_{p}^{'})$
such that $f_{n}\to f$ in $L^{0}(\mathbf{m})$ and $d_{\mathcal{F}_{p}}f_{n}\to d_{\mathcal{F}_{p}}f$
in $\mathcal{M}$ then $D(\mathcal{F}_{p}^{'})$ generates $\mathcal{M}_{\mathcal{F}_{p}}$.\end{cor}
\begin{proof}
The lemma is a consequence of the construction: As $\|[f,A]_{\mathcal{F}_{p}^{'}}\|_{\mathcal{F}_{p}^{'}}=\|[f,A]_{\mathcal{F}_{p}}\|_{\mathcal{F}_{p}}$
for all $f\in D(\mathcal{F}_{p}^{'})$ the assignment 
\[
i:[f,A]_{\mathcal{F}_{p}^{'}}\mapsto[f,A]_{\mathcal{F}_{p}}
\]
for $f\in D(\mathcal{F}_{p})$ is an isometric embedding which extends
uniquely to the module $\mathcal{M}_{\mathcal{F}_{p}^{'}}$ implying
that its images $i(\mathcal{M}_{\mathcal{F}_{p}^{'}})$ is a submodule
of $\mathcal{M}_{\mathcal{F}_{p}}$. 

To see the corollary, just note $d_{\mathcal{F}_{p}}f_{n}\to d_{\mathcal{F}_{p}}f$
also implies $\chi_{A}d_{\mathcal{F}_{p}}f_{n}\to\chi_{A}d_{\mathcal{F}_{p}}f=[f,A]_{\mathcal{F}_{p}}$.
Thus the generating set $\{[f,A]_{\mathcal{F}_{p}}\,|\, f\in D(\mathcal{F}_{p})\}$
of $\mathcal{M}_{\mathcal{F}_{p}}$ is a subset of $\mathcal{M}_{\mathcal{F}_{p}^{'}}$
which shows that the two modules agree. 
\end{proof}

\section{relaxed functionals and their modules}

In this section we present a general relaxation procedure of subadditive
$L^{p}$-valued functionals which fits into the framework of generalized
gradients. Furthermore, we give a representation of the modules generated
by functional and its relaxation. The construction here is more general
and could be simplified if we look at the local Lipschitz constants
(see next section). However, some of the results might be of interest
in other contexts. 

In the following assume $p\in(1,\infty)$ which implies reflexivity
of $L^{p}(\mathbf{m})$.

\subsection*{A relaxation procedure}

The idea of the relaxation procedure is get a version of the construction
in \cite[Section 4]{Ambrosio2013}. In particular, we are interested
in a ``pointwise'' minimality property of the relaxation. This is
known to hold for the (minimal) relaxed slope (or equivalently the
minimal weak upper gradient). 

The technique presented here shares similarities with the localization
method in the theory of $\Gamma$-convergence. As we are only looking
at $L^{p}$-valued functional we obtain a very precise description
of the lower relaxation via Lemma \ref{lem:weak-to-strong} and Theorem
\ref{thm:loc-approx-loc}.
\begin{rem*}
The choice of $L^{0}(\mathbf{m})$ is not essential. Any topological
vector space satisfying the conclusion of Lemma \ref{lem:convex-a.e.}
will do. For sake of naming let's call this property \emph{topological
Mazur property}. Any locally convex topological vector space, satisfies
this property, in particular, any Banach and Frechet space. 
\end{rem*}
Let $\mathcal{F}_{p}:D(\mathcal{F}_{p})\to L_{+}^{p}(\mathbf{m})$
be subadditive and absolutely homogenous, i.e. 
\begin{eqnarray*}
\mathcal{F}_{p}(f+g) & \le & \mathcal{F}_{p}(f)+\mathcal{F}_{p}(g)\\
\mathcal{F}_{p}(\lambda f) & \le & |\lambda|\mathcal{F}_{p}(f)
\end{eqnarray*}
for all $f,g\in D(\mathcal{F}_{p})$ and $\lambda\in\mathbb{R}$.
We extend $\mathcal{F}_{p}$ outside of $D(\mathcal{F}_{p})\subset L^{0}(\mathbf{m})$
by setting it to $\infty$. In general $D(\mathcal{\mathcal{F}}_{p})$
is not dense and the functional $E\mathcal{F}_{p}$ not lower semicontinuous
where 
\[
E\mathcal{F}_{p}(f)=\begin{cases}
\int g_{f}^{p}d\mathbf{m} & f\in D(\mathcal{F}_{p})\\
\infty & \mbox{otherwise.}
\end{cases}
\]

Define the following lower semi-continuous functional 
\[
E\check{\mathcal{F}}_{p}(f)=\inf\{\liminf_{n\to\infty}\int g_{f_{n}}^{p}d\mathbf{m}\,|\, D(\mathcal{F})\ni f_{n}\to f\,\mbox{ in }L^{0}(\mathbf{m})\}.
\]

This functional is called the lower semicontinuous hull of $E\mathcal{F}_{p}$.
Let $D(\check{\mathcal{F}}_{p})$ be the set of all $f\in L^{0}(\mathbf{m})$
with $E\check{\mathcal{F}}_{p}(f)<\infty$. We immediately obtain
the following characterization of the hull.
\begin{lem}
If $E:L^{0}(\mathbf{m})\to[0,\infty]$ is lower semicontinuous with
$E(f)\le D\mathcal{F}_{p}(f)$ then $E(f)\le E\check{\mathcal{F}_{p}}(f)$.
If, in addition, for every $f\in D(E)$ there is a sequence $D(\mathcal{F}_{p})\ni f_{n}\to f$
with $E(f)=\lim_{n\to0}\|g_{f}\|_{L^{p}}^{p}$ then $E(f)=E\check{\mathcal{F}}_{p}(f)$. 
\end{lem}
The following theorem shows that $E\check{\mathcal{F}}_{p}$ admits
an $L^{p}$-valued functional $\check{\mathcal{F}}_{p}f\mapsto\check{g}_{f}$
such that $E\check{\mathcal{F}}(f)=\|\check{g}_{f}\|_{L^{p}}^{p}$
hence justifying the notation.
\begin{prop}
[Relaxed representation]\label{prop:relaxed-function}For any $f\in D(\check{\mathcal{F}}_{p})$
there is a unique $\check{\mathcal{F}}_{p}(f)=\check{g}_{f}\in L^{p}(\mathbf{m})$
and a sequence $D(\mathcal{F}_{p})\in f_{n}\to f$ in $L^{0}(\mathbf{m})$
such that $g_{f_{n}}\to\check{g}_{f}$ strongly in $L^{p}(\mathbf{m})$
with $E\check{\mathcal{F}}_{p}(f)=\|\check{g}_{f}\|_{L^{p}}^{p}$
Furthermore, the assignment $f\mapsto\check{g}_{f}$ is subadditive
and absolutely homogenous.
\end{prop}
To prove this we need the following technical lemma.
\begin{lem}
\label{lem:convex-a.e.}If $f_{n}\to f$ in $L^{0}(\mathbf{m})$ then
for sequence $(g_{n})$ which is a (finite) convex combination of
$\{f_{k}\}_{k\ge n}$ also converges to $f$ in $L^{0}(\mathbf{m})$.\end{lem}
\begin{proof}
Let $\mathbf{m}(E)<\infty$. By Egorov's Theorem for every $\epsilon>0$
there is a measurable set $E_{\epsilon}$ such that $\mathbf{m}(E\backslash E_{\epsilon})$
and $f_{n}\to f$ uniformly on $E_{\epsilon}$. In particular, for
all $\delta>0$ there is an $n_{\delta}$ such that $|f_{n}(x)-f(x)|\le\delta$
for all $n\ge n_{\delta}$ and $x\in E_{\epsilon}$. Therefore, for
all (finite) convex combinations of $\{f_{k}\}_{k\ge n_{\epsilon}}$
it holds 
\[
|\sum_{k\ge n_{\delta}}\alpha_{k}f_{k}(x)-f(x)|\le\sum_{k\ge n_{\delta}}\alpha_{k}|f_{k}(x)-f(x)|\le\delta.
\]
But this implies $g_{n}=\sum_{k\ge n_{\delta}}\alpha_{k}f_{k}(x)$
converges uniformly to $f$ on $E_{\epsilon}$ for all $\epsilon$.
Thus $g_{n}\to g$ $\mathbf{m}$-almost everywhere on $E$ which shows
that $h_{n}\to h$ in $L^{0}(\mathbf{m})$.

[Proof of the Proposition]First let $f_{n}\to f$ with $E\check{\mathcal{F}}_{p}(f)=\lim_{n\to\infty}\int g_{f_{n}}^{p}d\mathbf{m}<\infty$.
Then one may replace $(g_{f_{n}})_{n\in\mathbb{N}}$ with a subsequence
and assume $(g_{f_{n}})_{n\in\mathbb{N}}$ converges weakly in $L^{p}(\mathbf{m})$
to some $g$. Mazur's Lemma implies that there is a sequence $g_{n}=\sum_{i=n}^{N(n)}\alpha_{i}^{n}g_{f_{i}}$
strongly converging in $L^{p}$ to $g$. By subadditivity of $f\mapsto g_{f}$
we also have $g_{h_{n}}\le g_{n}$ for $h_{n}=\sum_{i=n}^{N(n)}\alpha_{i}^{n}f_{i}$.
The previous lemma shows $h_{n}\to f$. But then 
\begin{eqnarray*}
E\check{\mathcal{F}}(f) & \le & \liminf_{n\to\infty}\int g_{h_{n}}^{p}d\mathbf{m}\le\limsup_{n\to\infty}\int g_{h_{n}}^{p}d\mathbf{m}\\
 & \le & \lim_{n\to\infty}\int g_{f_{n}}^{p}d\mathbf{m}=E\check{\mathcal{F}}(f)
\end{eqnarray*}
which implies that $\|g_{f_{n}}\|_{L^{p}}\to\|g\|_{L^{p}}$ and therefore
$g_{f_{n}}\to g$ strongly in $L^{p}(\mathbf{m})$. We claim that
$g$ is unique.

So assume for some sequences $f_{n},f_{n}^{'}\to f$ it holds $g_{f_{n}}\to g$
and $g_{f_{n}^{'}}\to g'$ with $D\check{\mathcal{F}}(f)=\|g\|_{L^{p}}^{p}=\|g'\|_{L^{p}}^{p}$.
Then for $h_{n}=\frac{f_{n}+f_{n}^{'}}{2}$ we have $g_{h_{n}}\le\frac{1}{2}(g_{f_{n}}+g_{f_{n}^{'}})$
so that uniform convexity of the $L^{p}$-norm implies $g=g'$ which
implies uniqueness.
\end{proof}
Taking the hull of function that is lower semicontinuous on its domain
just increases increases its domain to the maximal possible one. By
the strong approximation of $\check{g}_{f}$ we obtain the following.
\begin{cor}
If $E\mathcal{F}_{p}$ is lower semicontinuous on $D(\mathcal{F}_{p})$
then $g_{f}=\check{g}_{f}$ for all $f\in D(\mathcal{F}_{p})$. In
case $E\mathcal{F}_{p}$ is everywhere lower semicontinuous then $D(\check{\mathcal{F}}_{p})=D(\mathcal{F}_{p})$. 
\end{cor}
An equivalent characterization of $\check{g}_{f}$ can be given as
follows (compare \cite[Definition 4.1]{Ambrosio2013}): We say that
$G$ is an \emph{upper} \emph{relaxation} at $f$ if there is a sequence
$f_{n}\to f$ in $L^{0}(\mathbf{m})$ with $g_{f_{n}}\rightharpoonup\tilde{G}$
weakly in $L^{p}(\mathbf{m})$ and $\tilde{G}\le G$ $\mathbf{m}$-almost
everywhere. Denote the set of upper relaxations at $f$ by $\check{G}_{f}$.
One can show that the set of upper relaxations at $f$ is convex and
closed. If it is non-empty then uniform convexity of the $L^{p}$-norm
implies that there is a unique element of minimal norm. By the proposition
above this is given as the strong limit $\check{g}_{f}$ of $(g_{f_{n}})$
of a sequence $D(\mathcal{F}_{p})\ni f_{n}\to f$ in $L^{0}(\mathbf{m})$.
We call $\check{g}_{f}$ the \emph{minimal upper relaxation}. 

Proposition \ref{prop:relaxed-function} can be generalized as follows.
\begin{lem}
\label{lem:weak-to-strong}Assume $f_{n}\to f$ in $L^{p}(\mathbf{m})$
with $f_{n}\in D(\mathcal{F}_{p})$ and $g_{f_{n}}\rightharpoonup g$
weakly in $L^{p}(\mathbf{m})$. Then there is a sequence $(h_{n})_{n\in\mathbb{N}}$
in $D(\mathcal{F}_{p})$ converging in $L^{0}(\mathbf{m})$ to $f$
such that $g_{h_{n}}\to g'\le g$ strongly in $L^{p}(\mathbf{m})$.
Furthermore, $g'$ is unique and $g_{h_{n}^{'}}\to g'$ for any finite
convex combination $h_{n}^{'}$ of $\{h_{k}\}_{k\ge n}$.\end{lem}
\begin{rem*}
The uniqueness is w.r.t. the the sequence $f_{n}$. It may happen
that $g_{f_{n}^{1}},g_{f_{n}^{2}}\rightharpoonup g$ with $g_{1}^{'}\ne g_{2}^{'}$.\end{rem*}
\begin{proof}
Let $\mathcal{I}_{n}\subset2^{\mathbb{N}}$ such that $I\in\mathcal{I}_{n}$
iff $|I|<\infty$ and $\min I\ge n$. Note that $\mathcal{I}_{n}$
is countable and $\mathcal{I}_{n}\supset\mathcal{I}_{n'}$ for $n\le n'$.
Define 
\[
\mathcal{A}_{n}=\{(\alpha_{i})_{i\in\mathbb{N}}\,|\,\alpha_{i}\ge0,\sum\alpha_{i}=1,\alpha_{i}\ne0,\exists I\subset\mathcal{I}_{n}:\alpha_{i}\ne0\Rightarrow i\in I\}
\]
and set
\[
M=\inf\{\liminf_{n\to0}\int g_{h_{n}}^{p}d\mathbf{m}\,|\, h_{n}=\sum_{i\in\mathbb{N}}\alpha_{i}^{n}f_{n}\,\mbox{ for some }(\alpha_{i}^{n})_{i\in\mathbb{N}}\in\mathcal{A}_{n}\}.
\]
Note that $g_{h_{n}}$ is well-defined by subadditivity and homogeneity. 

The definition of $M$ yields a sequence $(\alpha_{i}^{n})_{i\in\mathbb{N}}\in\mathcal{A}_{n}$
such that 
\[
M=\lim_{n\to\infty}\int g_{h_{n}}^{p}d\mathbf{m}.
\]
In particular, $(g_{h_{n}})_{n\in\mathbb{N}}$ is bounded in $L^{p}(\mathbf{m})$
so after choosing a subsequence and relabeling assume $g_{h_{n}}\rightharpoonup g'$
weakly in $L^{p}(\mathbf{m})$. We claim $M=\|g'\|_{L^{p}}^{p}$.
Indeed, this claim would yield that $g_{h_{n}}\to g'$ strongly so
that $h_{n}$ is the required sequence. The proof below will also
show that any finite convex combination of the tails of $(h_{n})_{n\in\mathbb{N}}$
has the same property.

Let $(\beta_{i}^{n})_{i\in\mathbb{N}}\in\mathcal{A}_{n}$ be given
by Mazur's Lemma applied to $(g_{h_{n}})$, i.e. $\sum_{i}\beta_{i}^{n}g_{h_{i}}\to g'$
strongly in $L^{p}(\mathbf{m})$. Then there is a sequence $(\alpha_{i}^{n})_{i\in\mathbb{N}}\in\mathcal{A}_{n}$
such that $\sum_{i}\beta_{i}^{n}h_{i}=\sum\alpha_{i}^{n}f_{n}=\tilde{h}_{n}$
and by subadditivity 
\[
g_{\tilde{h_{n}}}\le\sum_{i\in\mathbb{N}}\beta_{i}^{n}g_{h_{i}}.
\]
Putting those facts together and use  convexity of $r\mapsto|r|^{p}$
yields 
\begin{eqnarray*}
M & \le & \liminf_{n\to\infty}\int g_{\tilde{h}_{n}}^{p}d\mathbf{m}\\
 & \le & \lim_{n\to\infty}\int\sum_{i\in\mathbb{N}}(\beta_{i}^{n}g_{h_{i}})^{p}d\mathbf{m}\\
 & \le & \lim_{n\to\infty}\sum_{i\in\mathbb{N}}\beta_{i}^{n}\int g_{h_{i}}^{p}d\mathbf{m}=M.
\end{eqnarray*}
As $\lim_{n\to\infty}\int\sum_{i\in\mathbb{N}}(\beta_{i}^{n}g_{h_{i}})^{p}d\mathbf{m}=\int(g')^{p}d\mathbf{m}$
we proved the claim. Observe that the proof also shows that for any
$g_{h_{n}^{'}}$ with $h_{n}^{'}$ being a finite convex combination
of $\{h_{k}\}_{k\ge n}$ converges strongly to $g'$ as well. 

It remains to show uniqueness. Assume $g_{1}^{'}$ and $g_{2}^{'}$
are obtained by some sequences $(h_{n}^{1})_{n\in\mathbb{N}}$ and
$(h_{n}^{2})_{n\in\mathbb{N}}$. By subadditivity $g_{h_{n}^{3}}\le g_{h_{n}^{1}}+g_{h_{n}^{2}}$
where $h_{n}^{3}=\frac{1}{2}(h_{n}^{1}+h_{n}^{2})$. Then any weak
accumulation point $g_{3}^{'}$ of $(g_{h_{n}^{3}})_{n\in\mathbb{N}}$
must satisfy $g_{3}^{'}\le\frac{1}{2}(g_{1}^{'}+g_{2}^{'})$. An argument
as above shows $(g_{h_{n}^{3}})$ converges strongly to $g_{3}^{'}$
with $\|g_{3}^{'}\|_{L^{p}}^{p}=M$. But $\|g_{1}^{'}\|=\|g_{2}^{'}\|=M$
so that strict convexity of the $L^{p}$-norm we show $g_{1}^{'}=g_{2}^{'}=g_{3}^{'}$.
\end{proof}
As it turns out the set of $g'$ obtained via the lemma are essential
for the characterization of relaxed modules defined below.
\begin{defn}
[strong upper relaxation] Let $\check{G}_{f}^{s}$ be the set of
$G\in\check{G}_{f}$ such that there is a sequence $f_{n}\to f$ in
$L^{0}(\mathbf{m})$ so that $g_{h_{n}}\to G$ strongly in $L^{p}(\mathbf{m})$
for every finite convex combination $h_{n}$ of $\{f_{k}\}_{k\ge n}$. 
\end{defn}
By the previous lemma $\check{G}_{f}^{s}$ is non-empty whenever $\check{G}_{f}$
is as $\check{g}_{f}\in\check{G}_{f}^{s}$. Furthermore, every $G\in\check{G}_{f}$
admits a strong relaxation $G'$ with $G'\le G$. 
\begin{rem*}
One may think of $g_{f}$ as the norm of some linear assignment to
a vectors like the gradient/differential of $f$. If an upper relaxation
represents the norm of a vector obtained via a closure procedure of
that assignment then it needs to be invariant under taking limits
of finite convex combinations of tail of the limiting construction.
In particular, their norm needs to satisfy this property. Hence the
only upper relaxations which can satisfy this property are strong
upper relaxations.
\end{rem*}
A priori it is not clear whether $f_{n}\to f$ with $\check{g}_{f_{n}}\rightharpoonup G$
implies $\check{g}_{f}\le G$ $\mathbf{m}$-almost everywhere. This
would be true if $g_{f}$ is the local Lipschitz constant and the
metric measure space is locally finite (see below). As the codomain
of $\mathcal{F}_{p}$ is $L^{p}(\mathbf{m})$, those functions enjoys
certain local properties, like integrability and convergence w.r.t.
subsets. An appropriate relaxation procedure should take this into
account. However, the relaxation procedure above only requires uniform
convexity of the $L^{p}$-norm and is therefore ``only'' a relaxation
of $D\mathcal{F}_{p}$. In the following we give an alternative construction
which satisfies this property and fits nicely into the module framework. 

Let $\hat{G}_{f}$ be the set of all $G\in L^{p}(\mathbf{m})$ such
that the following holds: Whenever $f_{n}\to f$ in $L^{0}(\mathbf{m})$
and $g_{f_{n}}\rightharpoonup\tilde{G}$ weakly in $L^{p}(\mathbf{m})$
then $G\le\tilde{G}$. It is easy to see that $\hat{G}_{f}$ is a
closed convex set and if $G_{1},G_{1}\in\hat{G}_{f}$ then also $\max\{G_{1},G_{2}\}\in\hat{G}_{f}$.
Furthermore, the set is always non-empty as it contains the zero element.
It is bounded if the defining condition is non-trivial, i.e. there
is a sequence $f_{n}\to f$ with $(g_{f_{n}})$ bounded in $L^{p}(\mathbf{m})$.

If $\hat{G}_{f}$ is bounded we claim that it has a unique element
of maximal $L^{p}$-norm which is given
\[
\hat{g}_{f}=\max_{G\in\hat{G}_{f}}G.
\]
To see that $\hat{g}$ is measurable and thus well-defined let $G_{n}$
be a sequence such that $\lim_{n\to\infty}\|G_{n}\|_{L^{p}}=\sup_{G\in\hat{G}_{f}}\|G\|_{L^{p}}$.
Replacing $G_{n}$ by $\tilde{G}_{n}=\max_{i=1}^{n}G_{i}$ we may
assume $G_{n}\le G_{n+1}$. By the monotone convergence theorem $G(x)=\lim_{n\to\infty}G_{n}(x)$
is measurable and 
\[
\sup_{G'\in\hat{G}_{f}}\|G'\|_{L^{p}}^{p}=\lim\int G_{n}^{p}d\mathbf{m}=\int G^{p}d\mathbf{m}.
\]
Assume there is $\tilde{G}\in\hat{G}_{f}$ such that $\tilde{G}>G$
on some set of positive measure. Then 
\[
\int\max\{G,\tilde{G}\}^{p}d\mathbf{m}>\int G^{p}d\mathbf{m}
\]
which contradicts the definition of $G$. Thus $\hat{g}_{f}=G$. Note
that we only required $f\in D(\check{\mathcal{F}}_{p})$. 

We call $\hat{g}_{f}$ the \emph{lower relaxation} $\mathcal{F}_{p}$
at $f\in D(\check{\mathcal{F}}_{p})$. For notational purposes also
write set $D(\hat{\mathcal{F}}_{p})=D(\check{\mathcal{F}}_{p})$ and
$\hat{\mathcal{F}}_{p}(f)=\hat{g}_{f}$.

From the characterization of the $\hat{g}_{f}$ and $\check{g}_{f}$
we obtain the following (compare AGS-Sect 4).
\begin{cor}
\label{cor:lower-upper-relax}The lower relaxation equals a.e. the
pointwise minimum of all upper relaxations at $f$. Therefore, if
whenever $g_{f_{n}}\rightharpoonup G$ it holds $\check{g}_{f}\le G$,
then $\hat{g}_{f}=\check{g}_{f}$.\end{cor}
\begin{rem*}
A sufficient condition for the above is that for any Borel set $B$
and upper relaxations $G,G'$ at $f$ also $\chi_{B}G+\chi_{M\backslash B}G'$
is an upper relaxation at $f$.
\end{rem*}
Obviously $f\mapsto\hat{g}_{f}$ is absolutely homogenous. To see
subadditivity note that it holds 
\[
g_{f_{n}+h_{n}}\le g_{f_{n}}+g_{h_{n}}.
\]
If now $(g_{f_{n}},g_{h_{n}})$ weakly to $(G_{1},G_{2})$ then any
weak limit $G_{3}$ of $g_{f_{n}+h_{n}}$ must satisfy $G_{3}\le G_{1}+G_{2}$. 

From the definition we also see that 
\[
E\hat{\mathcal{F}}_{p}(f)=\begin{cases}
\int\hat{g}_{f}^{p}d\mathbf{m} & f\in D(\hat{\mathcal{F}}_{p})\\
\infty & \mbox{otherwise}
\end{cases}
\]
is lower semicontinuous and $\hat{g}_{f_{n}}\rightharpoonup g$ weakly
in $L^{p}(\mathbf{m})$ with $f_{n}\to f$ in $L^{0}(\mathbf{m})$
and $\|g\|_{L^{p}}^{p}=E\hat{\mathcal{F}}_{p}(f)$ implies $\hat{g}_{f_{n}}\to g$
strongly and $g=\hat{g}_{f}$. 

We conclude with local characterization of $\hat{g}_{f}$. 
\begin{thm}
[Local approximation]\label{thm:loc-approx-loc}For any $\epsilon>0$
and $f\in D(\hat{\mathcal{F}}_{p})$ there is a Borel partition $\{A_{n}\}\cup A_{\infty}$
with $\mathbf{m}(A_{\infty})=0$ and sequence $(G_{n})$ of (strong)
upper relaxations at $f$ such that 
\[
\sum_{i\in\mathbb{N}}\chi_{A_{n}}G_{n}-f_{\epsilon}\le\hat{g}_{f}
\]
where $f_{\epsilon}$ is some non-negative $L^{p}$-integrable function
with $\|f_{\epsilon}\|_{L^{p}},\|f_{\epsilon}\|_{L^{\infty}}\le\epsilon$.
If $\hat{g}_{f}=\check{g}_{f}$ we can simply choose $A_{1}=M$,$G_{1}=\check{g}_{f}$
and $f_{\epsilon}\equiv0$.\end{thm}
\begin{proof}
Let $\{G_{n}\}_{n\in\mathbb{N}}$ be a countable dense set of $\check{G}_{f}$.
 Then as above let $\tilde{G}_{n}=\max_{k=1}^{n}G_{n}$ so that $\tilde{G}_{n}\ge\tilde{G}_{n+1}$.
(Note $\tilde{G}_{n}\in\check{G}_{f}$ would yield $\hat{g}_{f}=\check{g}_{f}$).
From the definition $G=\lim_{n\to\infty}\tilde{G}_{n}$ is measurable.
As $G^{p}\le G_{1}^{p}$ with $G_{1}^{p}\in L^{1}(\mathbf{m})$, the
Dominated Convergence Theorem implies 
\[
\int G^{p}d\mathbf{m}=\lim_{n\to\infty}\int\tilde{G}_{n}^{p}d\mathbf{m}
\]
and also $\tilde{G}_{n}\to G$ strongly in $L^{p}(\mathbf{m})$. 

It remains to show that $G\le G'$ for any $G'\in\check{G}_{f}$.
Suppose by contradiction there is a set of positive measure $A$ and
a $\delta>0$ with $G\ge G'+\delta$. Since $\{G_{n}\}_{n\in\mathbb{N}}$
is dense in $\check{G}_{f}$ we may find $G_{n}\to G'$ strongly in
$L^{p}(\mathbf{m})$. In particular, there is a subset $A'\subset A$
such that $(G_{n})_{|A'}-G_{|A'}^{'}\le\frac{\delta}{2}$ on $x\in A'$.
However, we have $G_{n}\ge G$ so that 
\[
\delta\le G_{|A'}-G_{|A'}^{'}\le(G_{n})_{|A'}-G_{|A'}^{'}\le\frac{\delta}{2}
\]
which is a contradiction. This shows $\hat{g}_{f}=G$.

The proof is finished by a standard approximation procedure: Let $A_{0}\subset M$
be a set of positive finite measure. As $\tilde{G}_{n}\to\hat{g}_{f}$
strongly in $L^{p}(\mathbf{m})$ we can also assume it converges $\mathbf{m}$-almost
everywhere on $A_{0}$. Then Egorov's Theorem implies that for each
$\delta>0$ and $\epsilon$ there is a $A_{1}\subset A_{0}$ and $n_{\epsilon}$
such that $\mathbf{m}(A_{0}\backslash A_{1})<\delta\mathbf{m}(A)$
and $\chi_{A_{1}}(\tilde{G}_{n_{\epsilon}}-\hat{g}_{f})\le\epsilon$.
But then there is a finite Borel partition $\{A_{1,k}\}_{k}$ of $A_{1}$
and indices $\{n_{1,k}\}$ such that $\chi_{A_{1,k}}(G_{n_{1,k}}-\hat{g})$.
In particular, the result holds on $A_{1}$. W.l.o.g. we may replace
each $G_{n_{1,k}}$ with its strong upper relaxation obtain from Lemma
\ref{lem:convex-a.e.}.

Repeating this argument we obtain a sequence $\{A_{n}\}_{n\in\mathbb{N}}$
with $A_{n+1}\subset A_{0}\backslash A_{n}$, partitions $\{A_{n,k}\}_{k=1}^{k_{n}}$
and indices $\{n_{n,k}\}_{k=1}^{k_{n}}$. Such that the above holds
on $\cup_{n\in\mathbb{N}}A_{n}$. Note that the construction shows
that $\mathbf{m}(A_{0}\backslash\cup A_{n})=0$ and $A_{n}\cap A_{n'}=\varnothing$
for $n\ne n'$ so that $\{A_{n}\}_{n\ge1}\cup\{A_{\infty}\}$ is the
require partition. 

If $\mathbf{m}(M)<\infty$ we are done. For the non-finite $\sigma$-finite
case let $\{A^{(k)}\}_{k\in\mathbb{N}}$ be a Borel partition of $M$
with $1\le\mathbf{m}(A^{(k)})<\infty$. On $A^{(k)}$ choose $\epsilon_{k}^{p}=\epsilon^{p}2^{-(k+1)}\mathbf{m}(A^{(k)})^{-1}$
and set $f_{\epsilon}=\sum\chi_{A^{(k)}}\epsilon_{k}$. Then $\|f_{\epsilon}\|_{L^{p}}=\epsilon\ge f_{\epsilon}$.
The result is now obtained by collecting the countable number of partitions
and indices.
\end{proof}

\subsection*{Characterization of the relaxed module}

Given $\mathcal{F}_{p}$ as above let $\hat{\mathcal{F}}_{p}$ be
its lower relaxation. Since it is also subadditive and absolutely
homogenous it induces the module $\mathcal{M}_{\hat{\mathcal{F}}_{p}}$
which we call the relaxed module of $\mathcal{M}_{\mathcal{F}_{p}}$.
Using Theorem \ref{thm:loc-approx-loc} we show that there is a general
relationship between $\mathcal{M}_{\mathcal{F}_{p}}$ and its relaxed
module. 

Recall $d_{\mathcal{F}_{p}}f=[f,M]_{\mathcal{F}_{p}}$. We will look
at the following set 
\[
\mathcal{G}_{0}=\{\omega\in\mathcal{M}_{\mathcal{F}_{p}}\,|\,\exists f_{n}\in D(\mathcal{F}_{p}):f_{n}\to0\mbox{ in }L^{0},d_{\mathcal{F}_{p}}f\to\omega\mbox{ in }\mathcal{M}_{\mathcal{F}_{p}}\}.
\]
One may verify that $\mathcal{G}_{0}$ is a closed subspace space
of $\mathcal{M}_{\mathcal{F}_{p}}$. In the following denote by $\mathcal{M}$,
$\hat{\mathcal{M}}$ and $\mathcal{M}_{0}$, the module $\mathcal{M}_{\mathcal{F}_{p}}$,
its relaxed module $\mathcal{M}_{\hat{\mathcal{F}}_{p}}$ and resp.
the submodule generated by $\mathcal{G}_{0}$. 

In general, it is not clear whether $\mathcal{G}_{0}$ contains more
than one element. However, if $\mathcal{G}_{0}$ is trivial the operator
$d_{\mathcal{F}_{p}}:D(\mathcal{F}_{p})\to\mathcal{M}$ is closable
as a (partially defined) linear operator from $L^{0}$ to $\mathcal{M}$
(compare also to \cite[Proposition 4.26]{Schioppa2014}). In general,
closability may not yield lower semicontinuity of $E\mathcal{F}_{p}$.
As it turns out it is possible to characterize the relationship of
the three modules if $\mathcal{M}$ satisfies a weak form of reflexivity.
First some technical results. 
\begin{lem}
[closedness]\label{lem:closedness}Assume $D(\hat{\mathcal{F}}_{p})\ni f_{n}\to f$
and $d_{\hat{\mathcal{F}}_{p}}f_{n}\to\omega\in\hat{\mathcal{M}}$
then $f\in D(\hat{\mathcal{F}}_{p})$ and $\omega=d_{\hat{\mathcal{F}}_{p}}f$.\end{lem}
\begin{proof}
The proof is just the abstract version of \cite[Theorem 2.2.9]{Gigli2014}:
The assumptions imply $\hat{g}_{f_{n}}\to|\omega|$ in $L^{p}$. As
$E\hat{\mathcal{F}}_{p}$ is lower semicontinuous we must have $f\in D(\hat{\mathcal{F}}_{p})$.
Again by lower semicontinuity we have
\[
\|d_{\hat{\mathcal{F}}_{p}}(f-f_{n})\|_{\hat{\mathcal{M}}}=\int g_{f-f_{n}}^{p}d\mathbf{m}\le\liminf_{m\to\infty}\|d_{\hat{\mathcal{F}}_{p}}(f_{m}-f_{n})\|_{\hat{\mathcal{M}}}.
\]
However, $d_{\hat{\mathcal{F}}_{p}}f_{n}$ is Cauchy in $\hat{\mathcal{M}}$
so that taking the $\limsup$ as $n\to0$, the right hand side converges
to zero. But this means $d_{\hat{\mathcal{F}}_{p}}f_{n}\to d_{\hat{\mathcal{F}}_{p}}f$,
i.e. $d_{\hat{\mathcal{F}}_{p}}f=\omega$.\end{proof}
\begin{lem}
[Mazur variant]\label{lem:improved-linear-Mazur}If $D(\mathcal{F})\ni f_{n}\to f$
in $L^{0}(\mathbf{m})$ and $d_{\mathcal{F}_{p}}f_{n}\rightharpoonup\omega$
weakly in $\mathcal{M}$ then there is a sequence $h_{n}\to f$ in
$L^{0}(\mathbf{m})$ such that $d_{\mathcal{F}_{p}}h_{n}\to\omega$
strongly and $h_{n}$ is a finite convex combination of $\{f_{k}\}_{k\ge n}$.
A similar result holds for $\hat{\mathcal{M}}$ and $\mathcal{M}_{0}$.\end{lem}
\begin{proof}
By Mazur's lemma $\omega$ is a strong limit a sequence $\omega_{n}$
which is a finite convex combination of $\{d_{\mathcal{F}_{p}}f_{k}\}_{k\ge n}$.
As $f\mapsto d_{\mathcal{F}_{p}}f$ is linear $\omega_{n}=[h_{n},M]$
where $h_{n}$ is a finite convex combination of $\{f_{k}\}_{k\ge n}$.
Lemma \ref{lem:convex-a.e.} shows that $h_{n}\to f$ in $L^{0}(\mathbf{m})$
proving the claim of this lemma.
\end{proof}
The following form of reflexivity will be needed. 
\begin{defn}
[weakly reflexive module] We say $\mathcal{M}$ is weakly reflexive
if any bounded sequence $d_{\mathcal{F}_{p}}f_{n}\in\mathcal{M}$
with $D(\mathcal{F}_{p})\ni f_{n}\to f$ admits a weakly convergent
subsequence.\end{defn}
\begin{question}
Is a weakly reflexive $L^{p}$-normed module reflexive?\end{question}
\begin{rem*}
This property and the question already appeared in \cite[Proposition 2.2.10, Remark 2.2.11]{Gigli2014}.
A possible argument could go as follows: For a Borel set $A$ define
\[
\mathcal{G}_{A}=\operatorname{cl}_{\|\cdot\|_{\mathcal{M}}}\{\chi_{A}d_{\mathcal{F}_{p}}f\,|\, f\in D(\mathcal{F}_{p})\}
\]
Let $\{B_{i}\}_{i\in\mathbb{N}}$ is a generating set of the Borel
$\sigma$-algebra then denote by $\{A_{i}^{n}\}_{i=1}^{N_{n}}$ the
partition generated by $\{B_{i}\}_{i=1}^{n}$. Then 
\[
\mathcal{M}_{n}=\bigoplus_{i=1}^{N_{n}}\mathcal{G}_{A_{i}^{n}}
\]
is reflexive if each $\mathcal{G}_{A_{i}^{n}}$ is. There are natural
embeddings$\mathcal{M}_{n}\to\mathcal{M}_{m}$ for $n\le m$, so that
$\mathcal{M}$ is a direct limit of $(\mathcal{M}_{n})$. So the question
is answered to the positive if $\mathcal{G}_{A}$ is reflexive for
each $A$ and any direct limit of reflexive Banach spaces is itself
reflexive. The difficulty in showing that $\mathcal{G}_{A}$ is reflexive
lies in the fact that $(d_{\mathcal{F}_{p}}f_{n})_{n\in\mathbb{N}}$
might be unbounded for some bounded sequence $(\chi_{A}d_{\mathcal{F}_{p}}f_{n})_{n\in\mathbb{N}}$. 
\end{rem*}
The next lemma shows that weak reflexivity implies that the strong
upper gradients are represented by an element in $\mathcal{M}$.
\begin{lem}
Assume $\mathcal{M}$ is weakly reflexive. Then for any strong upper
relaxation $G\in\check{G}_{f}^{s}$ there exist a $\omega\in\mathcal{M}$
with $|\omega|=G$ which is a limit of a sequence $(d_{\mathcal{F}_{p}}f_{n})_{n\in\mathbb{N}}$
with $f_{n}\to f$. If, in addition, $f\in D(\mathcal{F})$ then $(\omega-d_{\mathcal{F}_{p}}f)\in\mathcal{G}_{0}$.\end{lem}
\begin{proof}
Let $(f_{n})_{n\in\mathbb{N}}$ be as in the definition of strong
upper relaxation, i.e. $g_{h_{n}}\to G$ strongly in $L^{p}(\mathbf{m})$
where $h_{n}$ is finite convex combination of $\{f_{k}\}_{k\ge n}$. 

As this also holds for $(f_{n})_{n\in\mathbb{N}}$ we see that $(d_{\mathcal{F}_{p}}f_{n})_{n\in\mathbb{N}}$
is bounded in $\mathcal{M}$. Weak reflexivity shows there is subsequence
$d_{\mathcal{F}_{p}}f_{n'}\rightharpoonup\omega$. Applying the lemma
above to $(f_{n'})$ and $\omega$ gives a sequence $(h_{n})\subset D(\mathcal{F}_{p})$
with $d_{\mathcal{F}_{p}}h_{n}\to\omega$ and $g_{h_{n}}\to|\omega|$.
But then $|\omega|=G$. 
\end{proof}
The converse of that lemma is also true. If $d_{\mathcal{F}_{p}}f_{n}\to\omega$
and $f_{n}\to f$ then $|\omega|$ is a strong upper relaxations at
$f$. This implies an exact characterization of $\mathcal{G}_{0}$
in terms of strong upper relaxations. 
\begin{cor}
Assume $\mathcal{M}$ is weakly reflexive. Then $d_{\mathcal{F}_{p}}$
is closable iff there is at most one strong upper relaxation at $0$.
Or equivalently, $|\check{G}_{0}^{s}|=1$ iff $\mathcal{G}_{0}=\{0\}$. 
\end{cor}
Without reflexivity one can show at least the if-direction, see Theorem
\ref{thm:no-strong-upper} below.
\begin{thm}
[Relaxed representation]\label{thm:relaxed-rep}Assume $\mathcal{M}$
is weakly reflexive. Then 
\[
\hat{\mathcal{M}}'\cong\nicefrac{\mathcal{M}}{\mathcal{M}_{0}}
\]
where $\hat{\mathcal{M}}'$ is the submodule of $\hat{\mathcal{M}}$
generated by $D(\mathcal{F}_{p})\subset D(\hat{\mathcal{F}}_{p})$.
In particular, for each $f\in D(\mathcal{F}_{p})$ 
\[
\|d_{\hat{\mathcal{F}}_{p}}f\|_{\hat{\mathcal{M}}}=\inf_{\omega\in\mathcal{M}_{0}}\|d_{\mathcal{F}_{p}}f+\omega\|_{\mathcal{M}}.
\]
\end{thm}
\begin{rem*}
If the upper and lower relaxations at $f$ agree then the right hand
side is equal to $\inf_{\omega\in\mathcal{G}_{0}}\|d_{\mathcal{F}_{p}}f+\omega\|_{\mathcal{M}}$.\end{rem*}
\begin{proof}
It suffices to show 
\[
\|\chi_{A}d_{\hat{\mathcal{F}}_{p}}f\|_{\hat{\mathcal{M}}}=\inf_{\omega\in\mathcal{M}^{0}}\|\chi_{A}d_{\mathcal{F}_{p}}f+\omega\|_{\mathcal{M}}=\inf_{\omega\in\mathcal{M}^{0}}\|\chi_{A}d_{\mathcal{F}_{p}}f+\chi_{A}\omega\|_{\mathcal{M}}
\]
for all $f\in D(\mathcal{F}_{p})$ and Borel set $A\subset M$.

Let $\mathcal{P}$ be the set of all $\omega=\sum_{i\in\mathbb{N}}\chi_{A_{i}}\omega_{i}\in\mathcal{M}_{0}$
with $\{A_{i}\}_{i\in\mathbb{N}}$ a Borel partition of $M$ and $\omega_{i}\in G_{0}$.
The set $\mathcal{P}$ is dense in $\mathcal{M}^{0}$ so that we only
need to show the above for $\mathcal{P}$ instead of $\mathcal{M}_{0}$. 

From the definition of $\mathcal{G}_{0}$ there are $f_{i,n}\to0$
with $d_{\mathcal{F}_{p}}f_{i,n}\to\omega_{i}$ in $\mathcal{M}$.
Then also $d_{\mathcal{F}_{p}}(f+f_{i,n})\to\omega_{(i)}=d_{\mathcal{F}_{p}}f+\omega_{i}$
so that $g_{f+f_{i,n}}\to|\omega_{(i)}|$ in $L^{p}$. But then $\hat{g}_{f}\le|\omega_{(i)}|$
and thus 
\begin{eqnarray*}
\|\chi_{A}d_{\hat{\mathcal{F}}_{p}}f\|_{\mathcal{\hat{M}}} & = & \sum_{i\in\mathbb{N}}\int_{A\cap A_{i}}\hat{g}_{f}d\mathbf{m}\\
 & \le & \sum_{i\in\mathbb{N}}\int_{A\cap A_{i}}|\omega_{(i)}|^{p}d\mathbf{m}\\
 & = & \|\chi_{A}d_{\mathcal{F}_{p}}f+\chi_{A}\omega\|_{\mathcal{M}}^{p}.
\end{eqnarray*}

It remains to show that for each $\epsilon>0$ there is a $\omega_{\epsilon}\in\mathcal{P}$
such that 
\[
\|\chi_{A}d_{\mathcal{F}_{p}}f+\omega_{\epsilon}\|_{\mathcal{M}}\le\|\chi_{A}d_{\hat{\mathcal{F}}_{p}}f\|_{\hat{\mathcal{M}}}+\epsilon.
\]

Applying Theorem \ref{thm:loc-approx-loc} to $f$ we get a partition
$\{A_{i}\}_{i\in\mathbb{N}}$ and strong upper relaxations $G_{n}\in\check{G}_{f}^{s}$
such that 
\[
|\sum_{n\in\mathbb{N}}\chi_{A_{n}}G_{n}-\hat{g}_{f}|\le f_{\epsilon}
\]
with $\|f_{\epsilon}\|_{L^{p}}=\epsilon$. In particular,
\[
\|\sum\chi_{A\cap A_{n}}G_{n}\|_{L^{p}}\le\epsilon+\|\chi_{A}\hat{g}_{f}\|_{L^{p}}=\epsilon+\|\chi_{A}d_{\hat{\mathcal{F}}_{p}}f\|_{\hat{\mathcal{M}}}.
\]
It remains to show that $\|\chi_{A}\sum\chi_{A_{n}}G_{n}\|_{L^{p}}=\|\chi_{A}d_{\mathcal{F}_{p}}f+\omega_{\epsilon}\|_{\mathcal{M}}$
for some $\omega_{\epsilon}\in\mathcal{P}$. 

The lemma above shows that for each $G_{n}$ there is a $\omega_{n}\in\mathcal{M}$
such that $|\omega_{n}|=G_{n}$ and $\omega_{n}-d_{\mathcal{F}_{p}}f\in\mathcal{G}_{0}$.
Setting 
\[
\omega_{\epsilon}=\sum_{n\in\mathbb{N}}\chi_{A_{n}}(\omega_{n}-d_{\mathcal{F}_{p}}f)\in\mathcal{P}.
\]
shows 
\[
\|\chi_{A}d_{\mathcal{F}_{p}}f+\omega_{\epsilon}\|_{\mathcal{M}}=\|\sum_{n\in\mathbb{N}}\chi_{A\cap A_{n}}\omega_{n}\|_{\mathcal{M}}=\|\sum\chi_{A\cap A_{n}}G_{n}\|_{L^{p}}
\]
which proves the claim and thus the theorem. In case $\hat{g}_{f}=\check{g}_{f}$
one may choose $G_{1}=\hat{g}_{f}$ and $A_{1}=M$ and $\omega_{\epsilon}\in\mathcal{M}$
given by the lemma above is independent of $\epsilon>0$. \end{proof}
\begin{lem}
\label{lem:weak-reflex}If $\mathcal{M}$ is weakly reflexive and
the upper and lower relaxations agree then $\hat{\mathcal{M}}'$ is
weakly reflexive.\end{lem}
\begin{proof}
From proof above there is a sequence $(f_{m})_{m\in\mathbb{N}}$ in
$D(\mathcal{F}_{p})$ converging to $f$ such that $d_{\mathcal{F}_{p}}f_{m}\to\omega\in\mathcal{M}$
and 
\[
\|d_{\hat{\mathcal{F}}_{p}}f\|_{\hat{\mathcal{M}}}=\|d_{\mathcal{F}_{p}}f+\omega\|_{\mathcal{M}}.
\]
 Let $(f_{n})_{n\in\mathbb{N}}$ be a sequence in $D(\mathcal{F}_{p})$
converging in $L^{0}(\mathbf{m})$ to some $f\in D(\hat{\mathcal{F}}_{p})$
such that $(d_{\hat{\mathcal{F}}_{p}}f_{n})_{n\in\mathbb{N}}$ is
bounded. We may represent $d_{\hat{\mathcal{F}}_{p}}f_{n}$ by $d_{\mathcal{F}_{p}}f_{n}+\omega_{n}$
for some $\omega_{n}\in\mathcal{G}_{0}$. It suffices to show that
$d_{\mathcal{F}_{p}}f_{n}+\omega_{n}$ admits a weakly convergent
subsequence in $\mathcal{M}$.

For each $f_{n}$ there are sequences $(f_{n,m})_{m\in\mathbb{N}}$
in $D(\mathcal{F}_{p})$ converging to $f_{n}$ with $d_{\mathcal{F}_{p}}f_{n,m}\to d_{\mathcal{F}_{p}}f_{n}+\omega_{n}$
in $\mathcal{M}$ and $(d_{\mathcal{F}_{p}}f_{n,m})_{n,m\in\mathbb{N}}$
is bounded. Thus we may choose a diagonal sequence $g_{n}=f_{n,m_{n}}$
such that $g_{n}\to f$ and if for some subsequence $(g_{n'})$ it
holds $d_{\mathcal{F}_{p}}g_{n'}\rightharpoonup\omega$ then $d_{\mathcal{F}_{p}}f_{n'}+\omega_{n'}\rightharpoonup\omega$.
To conclude notice that $(d_{\mathcal{F}_{p}}g_{n})_{n\in\mathbb{N}}$
being a subsequence of $(d_{\mathcal{F}_{p}}f_{n,m})_{n,m\in\mathbb{N}}$
is bounded and admits a weakly convergent subsequence by weak reflexivity
of $\mathcal{M}$. \end{proof}
\begin{cor}
Assume $\mathcal{M}$ is reflexive or $\mathcal{M}$ is weakly reflexive
and the lower and upper relaxations agree. Then $D(\mathcal{F}_{p})$
generates $\hat{\mathcal{M}}$. In particular, 
\[
\hat{\mathcal{M}}\cong\nicefrac{\mathcal{M}}{\mathcal{M}_{0}}
\]
and $\hat{\mathcal{M}}$ is weakly reflexive.\end{cor}
\begin{proof}
In both cases $\hat{\mathcal{M}}'$ is weakly reflexive. Let $(f_{n})_{n\in\mathbb{N}}$
be a sequence in $D(\mathcal{F}_{p})$ converging in $L^{0}(\mathbf{m})$
to some $f\in D(\hat{\mathcal{F}}_{p})$ such that $(d_{\hat{\mathcal{F}}_{p}}f_{n})_{n\in\mathbb{N}}$
is bounded in $\hat{\mathcal{M}}'$. Such a sequence exists as $D(\hat{\mathcal{F}}_{p})=D(\check{\mathcal{F}}_{p})$.
Reflexivity and Lemma \ref{lem:improved-linear-Mazur} show that there
are finite convex combinations $h_{n}$ of $\{f_{k}\}_{k\ge n}$ such
that $h_{n}\to f$ in $L^{0}(\mathbf{m})$ and $d_{\hat{\mathcal{F}}_{p}}h_{n}\to\omega\in\hat{\mathcal{M}}'$
. However, $d_{\mathcal{F}_{p}}$ is closed so that $\omega=d_{\hat{\mathcal{F}}_{p}}f$.
We conclude using Lemma \ref{lem:subgenerator}. 
\end{proof}
A similar argument also shows the following vector space characterization.
We only sketch the argument and leave the details to the reader. Define
 
\[
\mathcal{G}=\operatorname{cl}_{\|\cdot\|_{\mathcal{M}}}\{d_{\mathcal{F}_{p}}f\in\mathcal{M}\,|\, f\in D(\mathcal{F}_{p})\}
\]
and 
\[
\check{\mathcal{G}}=\operatorname{cl}_{\|\cdot\|_{\check{\mathcal{M}}}}\{d_{\check{\mathcal{F}}_{p}}f\in\mathcal{M}\,|\, f\in D(\check{\mathcal{F}}_{p})\}
\]
 where $\check{\mathcal{F}}_{p}$ is the upper relaxation of $\mathcal{F}_{p}$
and $\check{\mathcal{M}}$ its induced module.
\begin{rem*}
All three sets $\mathcal{G}$, $\mathcal{G}_{0}$ and $\check{\mathcal{G}}$
only require the functional $E\mathcal{F}_{p}$ and its lower semicontinuous
envelope $E\check{\mathcal{F}}_{p}$. An explicit description of the
$L^{p}$-densities is not needed. The norm is then given by $(E\mathcal{F}_{p}(\cdot))^{\frac{1}{p}}$. \end{rem*}
\begin{prop}
\label{prop:vector-relaxed}If $\mathcal{G}$ is reflexive, i.e. $\mathcal{M}$
is weakly reflexive, then the following holds 
\[
\check{\mathcal{G}}\cong\nicerfrac{\mathcal{G}}{\mathcal{G}_{0}}
\]
and $d_{\check{\mathcal{F}}_{p}}:D(\mathcal{F}_{p})\to\check{\mathcal{G}}$
is closable such that its closure is given by $d_{\check{\mathcal{F}}_{p}}:D(\check{\mathcal{F}}_{p})\to\check{\mathcal{G}}$. 
\end{prop}
One may see this as part of a more general result: Assume $A:D\to W$
is a linear (unbounded) operator defined on some subset $D$ of a
topological vector space $V$ with topological Mazur property. Assume
$W$ is a reflexive Banach space and let $G$ be the closure of the
graph of $A$ in $V\times W$. Define $W_{0}$ to be the set of all
$w\in W$ with $(0,w)\in G$. It is easy to see that $W_{0}$ is a
closed subspace of $W$. Denote by $i:W\to\nicefrac{W}{W_{0}}$ the
quotient map. Then there is a uniquely defined closed linear operator
$\check{A}:\bar{D}\to\nicefrac{W}{W_{0}}$ such that $D\subset\bar{D}$
and $\check{A}v=i(Av)$ for $v\in D$.

\section{Lipschitz- and $L^{p}$-modules}

In this section we will use the above construction to obtain several
cotangent modules which will help to understand the analytic structure
of a metric space better. Assume again that $p\in(1,\infty)$.

\subsection*{Lipschitz and Cheeger energies}

The \emph{local Lipschitz constant }of a (local) Lipschitz function
$f$, also called slope of $f$, is defined as 
\[
\lip f(x)=\limsup_{y\to x}\frac{|f(y)-f(x)|}{d(y,x)}.
\]
It can be shown that $x\mapsto\lip f(x)$ is a measureable map and
if $\mathbf{m}$ has full support then the \emph{global Lipschitz
constant} 
\[
\Lip f=\sup_{x,y\in M}\frac{|f(y)-f(x)|}{d(y,x)}
\]
 is given by 
\[
\|\lip f\|_{L^{\infty}}=\Lip f.
\]
From the definition it is also easy to see that 
\[
\mathcal{F}_{p}^{\lip}:f\mapsto\lip f
\]
is subadditive and absolutely homogeneous. So denote by $D(\mathcal{F}_{p}^{\lip})$
the set of all $f\in\Lip(M,d)$ with $\lip f\in L^{p}(\mathbf{m})$. 

Let $\mathcal{F}_{p}$ be the lower relaxation of $\mathcal{F}_{p}^{\lip}$
. Denote lower relaxation at $f$ by $|Df|_{*,p}$ and called it the
\emph{relaxed slope}. 
\begin{defn}
[locally finite] The measure $\mathbf{m}$ is said to be \emph{locally
finite} if for each $x\in M$ there is an $r_{x}>0$ such that $\mathbf{m}(B_{r_{x}}(x))<\infty$.\end{defn}
\begin{prop}
Assume $\mathbf{m}$ is locally finite. Then the following holds: 
\begin{itemize}
\item the minimal upper and lower relaxations agree. 
\item For any Lebesgue null set $N\subset\mathbb{R}$ it holds $|Df|_{*,p}=0$
on $\mathbf{m}$-a.e. on $f^{-1}(N)$.
\item $|Df|_{*,p}=|Dg|_{*,p}$ $\mathbf{m}$-a.e. on $\{f=g\}$.
\item For any Lipschitz function $\varphi:\mathbb{R}\to\mathbb{R}$ it holds
$|D\varphi(f)|_{*,p}=|\varphi'(f)||Df|_{*,p}$ $\mathbf{m}$-a.e.
\end{itemize}
\end{prop}
\begin{proof}
With minor adjustments the proofs of \cite[Lemma 4.4 and 4.8]{Ambrosio2013}
still work if $f$ is assumed to be bounded. For equality in the last
case also see \cite[Equation (2.1.10)]{Gigli2014}. Setting $f_{N}=\min\{N,\max\{-N,f\}\}$
one may verify $|Df_{N}|_{*,p}\le|Df|_{*,p}$. Lower semicontinuity
implies $|Df_{N}|_{*,p}\to|Df|_{*,p}$ in $L^{p}$ which yields the
claims.\end{proof}
\begin{cor}
Assume $\mathbf{m}$ is locally finite. For every $f\in L^{p}(\mathbf{m})$
admitting a relaxed slope there is sequence $f_{n}\to f$ in $L^{p}(\mathbf{m})$
such that $\lip f_{n}\to|Df|_{*,p}$ in $L^{p}(\mathbf{m})$.\end{cor}
\begin{rem*}
As the relaxed slope obtained by $L^{p}$-approximations is a priori
larger then the one obtained by $L^{0}$-approximation, the statement
is equivalent to showing that the two notions agree. This was already
pointed out in \cite[Remark 4.4]{Ambrosio2011}.\end{rem*}
\begin{proof}
Denote by $|D^{*}f|_{*,p}$ the relaxed slope obtain by $L^{p}$-approximation
and assume $|Df|_{*,p}<|D^{*}f|_{*,p}$ on a set $A$ of positive
measure. W.l.o.g. assume $A$ is bounded. 

Let $f_{n}\to f$ in $L^{0}(\mathbf{m})$ with $\lip f_{n}\to|Df|_{*,p}$.
By Egorov's Theorem for every $\epsilon>0$ there is an $A_{\epsilon}$
such that $\mathbf{m}(A\backslash A_{\epsilon})<\epsilon$, and $f_{n}\to f$
and $\lip f_{n}\to|Df|_{*,p}$ uniformly on $A_{\epsilon}$ and each
of the function is uniformly bounded by $C$ on $A_{\epsilon}$. 

Using the MacShane's extension theorem we obtain a sequence $(g_{n})$
such that $(g_{n})_{|A}$ is Lipschitz with Lipschitz constants bounded
by $C$. Using a cut-off we can assume $(g_{n})_{n\in\mathbb{N}}$
is a sequence of Lipschitz functions with bounded Lipschitz constants
such that a neighborhood of their support has finite measure. In particular,
$(\lip g_{n})_{n\in\mathbb{N}}$ is uniformly bounded in $L^{p}(\mathbf{m})$.
So by reflexivity and Arzela-Ascoli we can assume $(g_{n})_{n\in\mathbb{N}}$
uniformly to some $g\in\Lip(M,d)$ with $g_{|A}=f_{|A}$ and $(\lip g_{n})_{n\in\mathbb{N}}$
converges weakly to some $G\in L^{p}(\mathbf{m})$. But then $|D^{*}g|_{*,p}\le G$
on $A$ as uniform convergence implies $L^{p}$-convergence. Also
note that $(G)_{|A}=(|Df|_{*,p})_{|A}$However, 
\begin{eqnarray*}
(|D^{*}f|_{*,p})_{|A} & = & (|D^{*}g|_{*,p})_{|A}\\
 & \le & G_{|A}\\
 & = & (|Df|_{*,p})_{|A}
\end{eqnarray*}
contradicting our assumption.
\end{proof}

\subsection*{Lipschitz and Sobolev modules}

We call the module generated by $\mathcal{F}_{p}^{\lip}$ the \emph{$L^{p}$-Lipschitz
module} and denote it by $L_{\lip}^{p}(T^{*}M)$. As $\mathcal{F}_{p}:f\mapsto|Df|_{*,p}$
is the lower relaxation of $\mathcal{F}_{p}^{\lip}$ it also generates
a $L^{p}(T^{*}M)$ which we call \emph{$L^{p}$-cotangent module}. 

In both cases the set of generator $D(\mathcal{F}_{p}^{\lip})$ and
$D(\mathcal{F}_{p})$ have a uniquely defined objects in those spaces.
We call those objects differentials as they agree with the usual notion
of a differential when $M$ is a smooth Riemannian/Finsler manifold.
\begin{defn*}
[Differentials] The \emph{($L^{p}$-)Lipschitz differential} of a
Lipschitz function $f\in D(\mathcal{F}_{p}^{\lip}$) is defined as
\[
d_{\lip}f=d_{\mathcal{F}_{p}^{\lip}}f\in L_{\lip}^{p}(T^{*}M).
\]
If $f\in D(\mathcal{F}_{p})$ then the \emph{$L^{p}$-differential}
(sometimes $L^{p}$-Sobolev differential) is defined as
\[
d_{p}f=d_{\mathcal{F}_{p}}f\in L^{p}(T^{*}M).
\]
\end{defn*}
\begin{rem*}
The differential $d_{\lip}$ does not really depend on $p$ as it
is uniquely defined only depending on the local Lipschitz constant
which needs to be in $L^{p}$. The $L^{p}$-differential, however,
depends in general on $p$ (see \cite{DiMarino2015}). 
\end{rem*}
The following is a result of the calculus of the relaxed slope. We
refer to \cite{Gigli2014} as we don't need those properties in the
course of this paper.
\begin{lem}
[$L^p$-calculus {\cite[Theorem 2.2.6]{Gigli2014}}] The operator $d_{p}$
satisfies the following:
\begin{itemize}
\item \noun{(Leibniz rule)} For any $f,g\in D(\mathcal{F}_{p})\cap L^{\infty}(\mathbf{m})$
it holds 
\[
d_{p}(fg)=f\cdot d_{p}g+g\cdot d_{p}f\qquad\mathbf{m}\mbox{-almost everywhere.}
\]

\item \noun{(Chain rule)} For any Lebesgue null set $N\subset\mathbb{R}$
and any $f\in D(\mathcal{F}_{p})$ 
\[
d_{p}f=0\qquad\mathbf{m}\mbox{-almost everywhere on }f^{-1}(N)
\]
and any open $I\subset\mathbb{R}$ such that $\mathbf{m}(f^{-1}(\mathbb{R}\backslash I))=0$
and any Lipschitz function $\varphi:I\to\mathbb{R}$ it holds
\[
d_{p}(\varphi\circ f)=\varphi'(f)\cdot d_{p}f\qquad\mathbf{m}\mbox{-almost everywhere.}
\]

\item \noun{(Locality)} For any $f\in D(\mathcal{F}_{p})$ it hold
\[
d_{p}f=d_{p}g\qquad\mathbf{m}\mbox{-almost everywhere on }\{f=g\}.
\]

\end{itemize}
\end{lem}
We have the following which is just Lemma \ref{lem:closedness} in
terms of the cotangent module.
\begin{prop}
[Closedness of $d_p$] Assume $f\in D(\mathcal{F}_{p})$ is converging
almost everywhere to a measurable function $f\in L^{0}(\mathbf{m})$.
If $d_{p}f_{n}\to\omega\in L^{p}(T^{*}M)$ then $f\in D(\mathcal{F}_{p})$
and $d_{p}f=\omega$. 
\end{prop}
Because $d_{p}$ is a closed operator we can define a Sobolev space
via the graph norm of $d_{p}$.
\begin{defn}
[Sobolev space] The Sobolev space $W^{1,p}(M,\mathbf{m})$ is defined
as the set of $L^{p}$-functions with $f\in D(\mathcal{F}_{p})$.
The norm of $W^{1,p}(M,\mathbf{m})$ is given by
\[
\|f\|_{W^{1,p}}^{p}=\|f\|_{L^{p}}^{p}+\|d_{p}f\|_{L^{p}(T^{*}M)}^{p}.
\]
By the previous preposition $W^{1,p}(M,\mathbf{m})$ is a Banach space.
\end{defn}
The operator $d_{\lip}$ is in general not closed as an operator from
$L^{0}$ to $L^{p}(T^{*}M)$. Nevertheless, one might ask whether
it is closable. From the above it is clear that $d_{\lip}$ is closable
if $\lip f=|Df|_{*,p}$ for all Lipschitz functions in $D(\mathcal{F}_{p})$.
In this case $D(\mathcal{F}_{p}^{\lip})$ also generates $\mathcal{M}_{\mathcal{F}_{p}}$.
If Lipschitz functions are dense in $W^{1,p}$ then by Lemma \ref{lem:subgenerator}
and its corollary we have $\mathcal{M}_{\mathcal{F}_{p}}=\mathcal{M}_{\mathcal{F}_{p}^{\lip}}$.
However, it is not clear if closability of $d_{\lip}$ would imply
lower semicontinuity of $f\mapsto\int(\lip f)^{p}d\mathbf{m}$ on
$D(\mathcal{F}_{p}^{\lip})$.

In case the $L^{p}$-Lipschitz module is weakly reflexive one can
show that closability of $d_{\lip}$ is equivalent to lower semicontinuity.
\begin{thm}
[$\text{Lipschitz}_0$ module]\label{thm:zero-Lipschitz-module} Assume
$L_{\lip}^{p}(T^{*}M)$ is weakly reflexive. Let $L_{\lip_{0}}^{p}(T^{*}M)$
be the set generated by $\omega\in L_{\lip}^{p}(T^{*}M)$ such that
there is a sequence $f_{n}\in\Lip(M,d)$ with $f_{n}\to0$ and $d_{\lip}f_{n}\to\omega$
strongly in $L_{\lip}^{p}(T^{*}M)$. Then $L_{\lip}^{p}(T^{*}M)$
is a submodule of $L_{\lip_{0}}^{p}(T^{*}M)$ such that 
\[
L^{p}(T^{*}M)\cong\nicerfrac{L_{\lip}^{p}(T^{*}M)}{L_{\lip_{0}}^{p}(T^{*}M)}.
\]
In addition, $d_{\lip}$ is closable as an operator from $L^{0}\to L_{\lip}^{p}(T^{*}M)$
iff $L_{\lip_{0}}^{p}(T^{*}M)$ is trivial in which case $L^{p}(T^{*}M)\cong L_{\lip}^{p}(T^{*}M)$
and $|Df|_{*,p}=\lip f$ for all $f\in\Lip(M,d)\cap L^{p}(T^{*}M)$.\end{thm}
\begin{proof}
See Theorem \ref{thm:relaxed-rep}.\end{proof}
\begin{cor}
If $L_{\lip}^{p}(T^{*}M)$ is weakly reflexive then $W^{1,p}(M,\mathbf{m})$
and $L^{p}(T^{*}M)$ are reflexive.\end{cor}
\begin{proof}
This could be deduced from Lemma \ref{prop:vector-relaxed} since
the lower and upper relaxations of the local Lipschitz constant agree.
We give a separate argument for $W^{1,p}(M,\mathbf{m})$. Define the
pseudo Sobolev space $W_{\lip}^{1,p}(M,\mathbf{m})$ as the closure
of all $L^{p}$-integrable Lipschitz functions in $D(\mathcal{F}_{p}^{\lip})$
where the norm (on $D(\mathcal{F}_{p}^{\lip})$) is given by 
\[
\|f\|_{W_{\lip}^{1,p}}^{p}=\|f\|_{L^{p}}^{p}+\|d_{\lip}f\|_{L_{\lip}^{p}(T^{*}M)}^{p}.
\]
 This space agrees with a certain Sobolev space defined in \cite{Franchi1999,Schioppa2014}. 

In the following we identify $W_{\lip}^{1,p}(M,\mathbf{m})$ as a
closed convex subset of $L^{p}(\mathbf{m})\times L_{\lip}^{p}(T^{*}M)$
equipped with the $L^{p}$-product norm and to keep the notation simple
set $\mathcal{V}=W_{\lip}^{1,p}(M,\mathbf{m})$. 

The closed subspace measures the lack closability of $d_{\lip}$ at
the origin is the following: 
\[
\mathcal{V}_{0}=\{v\in W_{\lip}^{1,p}\,|\,\exists f_{n}\in D(\mathcal{F}_{p}^{\lip})\cap L^{p}(\mathbf{m}):f_{n}\overset{L^{p}}{\longrightarrow}0,f_{n}\overset{W_{\lip}^{1,p}}{\longrightarrow}v\}.
\]
As a subspace of $L^{p}(\mathbf{m})\times L_{\lip}^{p}(T^{*}M)$ this
translate to $\mathcal{V}_{0}=\{0\}\times\mathcal{G}_{0}$.

We claim that $\mathcal{V}$ and thus $\mathcal{V}_{0}$ are reflexive:
If $(v_{n})_{n\in\mathbb{N}}$ is bounded in $\mathcal{V}$ then there
are Lipschitz functions $f_{n,m}\in\mathcal{V}$ such that $f_{n,m}\to v_{n}$.
In particular, choosing a diagonal sequence we obtain a sequence $(f_{n,m_{n}})_{n\in\mathbb{N}}\subset\mathcal{V}$
of Lipschitz functions such that $f_{n,m_{n}}\rightharpoonup v$ weakly
in $\mathcal{V}$ iff $v_{n}\rightharpoonup v$ weakly in $\mathcal{V}$.
Weak reflexivity of $L_{\lip}^{p}(T^{*}M)$ and reflexivity of $L^{p}(\mathbf{m})$
show that there are $f\in L^{p}(\mathbf{m})$ and $\omega\in L_{\lip}^{p}(T^{*}M)$
such that $(f_{n',m_{n'}},d_{\lip}f_{n',m_{n'}})\rightharpoonup(f,\omega)$.
But then there is a $v\in\mathcal{V}$ represented by $(f,\omega)$
with $f_{n',m_{n'}}\rightharpoonup v$ proving the claim.

As $|Df|_{*,p}$ is also an upper relaxation, the proof of Theorem
\ref{thm:relaxed-rep} shows
\[
\|d_{p}f\|_{L^{p}(T^{*}M)}=\inf_{\omega\in\mathcal{G}_{0}}\|d_{\lip}f+\omega\|_{L_{\lip}^{p}(T^{*}M)}
\]
for all Lipschitz functions $f\in W^{1,p}(M,\mathbf{m})$. But then
\[
W^{1,p}(M,\mathbf{m})\cong\nicerfrac{\mathcal{V}}{\mathcal{V}_{0}}
\]
which shows reflexivity of $W^{1,p}(M,\mathbf{m})$.
\end{proof}
It might well happen that the $L^{p}$-Lipschitz cotangent module
is not weakly reflexive but $d_{\lip}$ closable. By \cite[Corollary 7.5]{Ambrosio2012}
metrically doubling space have (super)reflexive Sobolev spaces. In
such a case it is not clear whether the modules agree and if the identification
can be given as above. Nevertheless, if there are no non-trivial upper
relaxations we obtain the following. 
\begin{thm}
\label{thm:no-strong-upper}Assume $\mathbf{m}$ is finite on bounded
sets. If $0$ admits no non-trivial strong upper relaxation of the
local Lipschitz constant at $0$ then $f\mapsto\int(\lip f)^{p}d\mathbf{m}$
is lower semicontinuous on $D(\mathcal{F}_{p}^{\lip})$, $d_{\lip}$
is closable and the $L^{p}$-Lipschitz module is a submodule of the
$L^{p}$-cotangent module. If Lipschitz differentials are dense in
the set of Sobolev differentials then the two modules agree. \end{thm}
\begin{proof}
The proof is similar to the proof of \cite[Theorem 8.4]{Ambrosio2012}.
Assume $f_{n}\to f$ in $L^{0}(\mathbf{m})$ with $f_{n},f\in D(\mathcal{F}_{p}^{\lip})$
and $(\lip f_{n})_{n\in\mathbb{N}}$ is bounded in $L^{p}(\mathbf{m})$.
By choosing a subsequence we may assume
\[
\liminf_{n\to\infty}\int(\lip f_{n})^{p}d\mathbf{m}=\lim_{n\to\infty}\int(\lip f_{n})^{p}d\mathbf{m}.
\]

Because $\lip(f_{n}-f)$ is also bounded in $L^{p}(\mathbf{m})$ we
may further replace $(f_{n})_{n\in\mathbb{N}}$ by one of its subsequence
and assume $\lip(f-f_{n})\rightharpoonup G$ weakly in $L^{p}(\mathbf{m})$.
Then Lemma \ref{lem:weak-to-strong} together with triviality of strong
upper gradients at $0$ shows there is a sequence $h_{n}\in D(\mathcal{F}_{p}^{\lip})$
of finite convex combinations of $\{f_{k}\}_{k\ge n}$ with $(f-h_{n})\to0$
in $L^{0}(\mathbf{m})$ and $(\lip(f-h_{n}))_{n\in\mathbb{N}}$ converges
strongly in $L^{p}(\mathbf{m})$ to some strong upper relaxation $G'\le G$
at $0$. But then $G'=0$. 

Observe that subadditivity of $f\mapsto\lip f$ implies 
\[
\int(\lip h_{n})^{p}d\mathbf{m}\le\sup_{k\ge n}\int(\lip f_{n})^{p}d\mathbf{m}.
\]
Therefore,
\begin{eqnarray*}
\left(\int(\lip f)^{p}d\mathbf{m}\right)^{\frac{1}{p}} & \le & \liminf_{n\to\mathbb{N}}\left\{ \left(\int(\lip h_{n})^{p}d\mathbf{m}\right)^{\frac{1}{p}}+\left((\lip(f-h_{n})^{p}d\mathbf{m}\right)^{\frac{1}{p}}\right\} \\
 & \le & \limsup_{n\to\mathbb{N}}\left(\int(\lip f_{n})^{p}d\mathbf{m}\right)^{\frac{1}{p}}+\lim_{n\to\infty}\left((\lip(f-h_{n})^{p}d\mathbf{m}\right)^{\frac{1}{p}}\\
 & \le & \lim_{n\to\mathbb{N}}\left(\int(\lip f_{n})^{p}d\mathbf{m}\right)^{\frac{1}{p}}.
\end{eqnarray*}
Thus the upper relaxation of $\mathcal{F}_{p}^{\lip}:f\mapsto\lip f$
is trivial on $D(\mathcal{F}_{p}^{\lip})$ and Lemma \ref{lem:subgenerator}
and its corollary imply the theorem.
\end{proof}
One can show that a lack of closability implies that there are non-trivial
strong upper relaxations at $0$. However, whether closability implies
absence of non-trivial upper relaxations is not clear. 
\begin{question}
Is there (compact/doubling) metric measure space such that $d_{\lip}$
is closable but $f\mapsto\int(\lip f)^{p}d\mathbf{m}$ not lower semicontinuous? 
\end{question}
Furthermore, one may ask if the lack of lower semicontinuity implies
that certain trivialities w.r.t. to Sobolev and Lipschitz functions.
\begin{question*}
Does $L_{\lip_{0}}^{p}(T^{*}M)\ne\{0\}$ imply that there is a non-trivial
Lipschitz function with $d_{p}f=0$? Vice versa, assume whenever $f\in\Lip(M,d)$
and $d_{p}f=0$ then $f\equiv\mathsf{const}$. Does this imply that
$L_{\lip_{0}}^{p}(T^{*}M)$ is trivial?
\end{question*}
The known (non-product) examples show that the space has either trivial
$L^{p}$-structure or the trivial $\mbox{Lipschitz}_{0}$-module.
An example of a space between these two extremes might help to answer
that question.

\subsection*{Lipschitz differential structure and Cheeger differentials}

In this section we combine the abstract theory above with the theory
of metric spaces admitting a Lipschitz differentiable structure. It
turns out that that theory fits nicely into the abstract structure
of the previous section. For further reference see\cite{Cheeger1999,Keith2004,Bate2014}. 

The following can be deduced from the fact that $L_{\lip}^{p}(T^{*}M)$
is generated by Lipschitz functions. 
\begin{lem}
There is a Borel partition $\{A_{\infty}\}\cup\{A_{i,n}\}_{i,n\in\mathbb{N}}$
of $M$ such that for each $i,n\in\mathbb{N}$ there are Lipschitz
functions $\{f_{1},\ldots,f_{n}\}$ that their Lipschitz differentials
generate $L_{\lip}^{p}(T^{*}A_{i,n})$. In particular, we can assume
$A_{\infty}=E_{\infty}$ and $A_{i,n}\subset E_{n}$ where $E_{n}$
is given by the dimensional decomposition of $L_{\lip}^{p}(T^{*}M)$. 
\end{lem}
By further partitioning $A_{i,n}$ we can assume each $A_{i,n}$ is
a bounded and each element $f_{i}$ of the basis has bounded support.
In particular, $d_{\lip}f_{i}\in L^{p}(T^{*}M)$ for all $p\in(1,\infty)$
if $\mathbf{m}(\supp f_{i})<\infty$. 
\begin{cor}
Assume $\mathbf{m}$ is finite on bounded subset. Then the dimensional
decomposition above is valid independent of $p\in(1,\infty)$. In
addition, $L_{\lip}^{p}(T^{*}M)$ is locally finite dimensional (bounded
by $N$) for some $p\in(1,\infty)$ iff it is locally finite dimensional
(bounded by $N$) for all $p\in(1,\infty)$. 
\end{cor}
Those results resemble exactly what is known ``in case $p=\infty$'':
The assignment $f\mapsto\lip f(x)$ is a semi-norm for all $x\in M$.
A finite set $\{f_{i}\}_{i=1}^{n}$ of Lipschitz functions is said
to be independent on a set $A$ if for $\mathbf{m}$-almost all $x\in A$
the assignment $f\mapsto\lip f(x)$ is a norm when restricted to $\operatorname{span}\{f_{i}\}_{i=1}^{n}$. 
\begin{defn}
[Differentiability space] A metric measure space $(M,d,\mathbf{m})$
is called a (Lipschitz) differentiability space if there is a uniform
bound $N$ on the maximal number of Lipschitz functions which can
be independent on a set of positive measure.
\end{defn}
By a cut-off argument the decomposition is local so that we obtain
the following.
\begin{cor}
A metric measure space which assigns finite measure to each bounded
set is a differentiability space with constant $N\in\mathbb{N}$ iff
$L_{\lip}^{p}(T^{*}M)$ is locally finite dimensional bounded the
same constant $N$ for some (and thus all) $p\in(1,\infty)$.
\end{cor}
Choose a basis $\{f_{1}^{i,n},\ldots,f_{n}^{i,n}\}$ of Lipschitz
functions on each $A_{i,n}$ such that for each Lipschitz map $f$
there is a linear map $f\mapsto\alpha_{f,A_{i,n}}^{j}\in L^{\infty}(A_{i,n})$,
$j=1,\ldots,n$ with $d_{\lip}f=\sum_{j=1}^{n}\alpha_{f,A_{i,n}}^{j}d_{\lip}f_{j}^{i,n}$.
The definition implies for $\mathbf{m}$-almost all $x_{0}\in A_{i,n}$
\[
\lip(f-\sum_{j=1}^{n}\alpha_{j,A_{i,n}}^{j}(x_{0})f_{j}^{i,n})(x_{0})=0
\]
or equivalently
\[
f(x)=f(x_{0})+\sum_{j=1}^{n}\alpha_{j,A_{i,n}}^{j}(x_{0})(f_{j}^{i,n}(x)-f_{j}^{i,n}(x_{0}))+o(d(x,x_{0})).
\]
 One may verify that $\lip f(x_{0})=\lip\left(\sum_{j=1}^{n}\alpha_{j,A_{i,n}}^{j}(x_{0})(f_{j}^{i,n}(\cdot)-f_{j}^{i,n}(x_{0}))\right)(x_{0})$. 

We call the operator 
\[
D:f\mapsto\sum_{i,n\in\mathbb{N}}\chi_{A_{i,n}}\boldsymbol{\alpha}_{f,A_{i,n}}\in\bigoplus_{n\in\mathbb{N}}\mathbb{R}^{n}.
\]
the \emph{Cheeger differential} of $f$ (w.r.t. the structure induced
by $\{A_{i,n}\}$ and $\{f_{j}^{i,n}\}$). Note that this is just
a form of the representation theorem (Theorem \ref{thm:reflexive-renorm})
w.r.t. the chosen basis element. 
\begin{rem*}
The Lipschitz differentiable structure only depends on the measure
class of $\mathbf{m}$, i.e. if $\mathbf{m}'$ is a measure with $\mathbf{m}'=f\mathbf{m}$
and $f(x)\in(0,\infty)$ $\mathbf{m}$-almost every then $(M,d,\mathbf{m}')$
is Lipschitz differentiable as well. This follows from the fact that
$\mathbf{m}(M\backslash\cup_{N>0}\Omega_{N})=0$ where $\Omega_{N}=\{x\in M\,|\, f(x)\in(N^{-1},N)\}$. 
\end{rem*}
Using the characterization theorem we get a more precise description
of the modules and their dimensions.
\begin{defn}
[Dimension]The \emph{$L^{p}$-Lipschitz dimension} $\dim_{\lip}A$
of a Borel set $A$ is defined as 
\[
\dim_{\lip}A=\esssup_{x\in A}\dim(L_{\lip}^{p}(T^{*}M),x).
\]
Similarly, define the \emph{$L^{p}$-analytic dimension} $\dim_{p}A$
of $A$ as the essential supremum of the $L^{p}$-cotangent module.\end{defn}
\begin{rem*}
If $\mathbf{m}$ is finite on bounded sets then $\dim_{\lip}$ does
not depend on $p$.
\end{rem*}
If the $L^{p}$-Lipschitz module is locally finite Theorem \ref{thm:reflexive-renorm}
yields.
\begin{cor}
If the $L^{p}$-Lipschitz module is locally finite dimensional (bounded
by $N$) then $L_{\lip}^{p}(T^{*}M)$ is reflexive and the above theorem
applies. Furthermore, 
\[
\dim_{\lip}(M,x)=\dim_{p}(M,x)-\dim(L_{\lip_{0}}^{p}(T^{*}M),x)\qquad\mbox{\ensuremath{\mathbf{m}}}\mbox{-almost everwhere}.
\]
 In particular, $L^{p}(T^{*}M)$ is locally finite dimensional (bounded
by $N$), and $L^{p}(T^{*}M)$ and $W^{1,p}(M,\mathbf{m})$ are reflexive.
Furthermore, $d_{\lip}$ is closable iff $\dim_{\lip}(M,x)=\dim_{p}(M,x)$
for $\mathbf{m}$-almost all $x\in M$ and both is equivalent to lower
semicontinuity of $f\mapsto\frac{1}{p}\int(\lip f)^{p}d\mathbf{m}$. \end{cor}
\begin{rem*}
The fact that a $L^{p}$-Lipschitz cotangent module, which is locally
finite dimensional bounded by $N$, is weakly reflexive was already
observed by Schioppa \cite[Theorem 4.16]{Schioppa2014}. \end{rem*}
\begin{proof}
Since any closed submodule of a locally finite dimensional $L^{p}(\mathbf{m})$-normed
module is also locally finite dimensional with the same bound, the
only thing to be proved is the dimension formula. For this assume
$E$ is a set where all dimensions are constant. Then 
\begin{eqnarray*}
L_{\lip}^{p}(T^{*}E) & \cong & L^{p}(E,\mathbb{R}^{n_{1}},|\cdot|)\\
L_{\lip_{0}}^{p}(T^{*}E) & \cong & L^{p}(E,\mathbb{R}^{n_{2}},|\cdot|)\\
L^{p}(T^{*}E) & \cong & L^{p}(E,\mathbb{R}^{n_{3}},|\cdot|^{'}).
\end{eqnarray*}
The theorem implies 
\begin{eqnarray*}
L^{p}(E,\mathbb{R}^{n_{3}},|\cdot|^{'}) & \cong & \nicerfrac{L^{p}(E,\mathbb{R}^{n_{1}},|\cdot|)}{L^{p}(E,\mathbb{R}^{n_{2}},|\cdot|)}.
\end{eqnarray*}
But then 
\[
(\mathbb{R}^{n_{3}},|\cdot|_{x}^{'})=\nicerfrac{(\mathbb{R}^{n_{1}},|\cdot|_{x})}{(\mathbb{R}^{n_{2}},|\cdot|_{x})}
\]
which implies $n_{3}=n_{1}-n_{2}$.
\end{proof}

\subsection*{Rigidity theorems of differentiability spaces}

We now state a result which was proven by \cite{Alberti2010,Alberti2014}
(see also \cite[Remark 6.11]{Bate2014}). Namely Lipschitz differentiable
measures in $\mathbb{R}^{n}$ are very rigid. We state their result
in the language of differentiability spaces and give a short argument.
\begin{prop}
[Rigidity Theorem I] If $(\mathbb{R}^{n},d_{\operatorname{Euclid}},\mathbf{m})$
is a differentiability space such then $\mathbf{m}$ is absolutely
continuous w.r.t. the Lebesgue measure on $\mathbb{R}^{n}$. \end{prop}
\begin{proof}
By \cite{Keith2004} one may choose as chart functions on $A_{i,k}$
the distance functions $\{d_{j}\}_{j=1}^{k}$ with $d_{j}(x)=\|x-x_{j}\|$
for some points $\{x_{1},\ldots,x_{k}\}$ such that $\mathbf{m}(\{x_{1},\ldots,x_{k}\})=0$.
As each $d_{j}$ is differentiable in the usual sense away from $\{x_{1},\ldots,x_{k}\}$
we may take $\{-\nicefrac{d_{j}^{2}}{2}\}_{j=1}^{k}$ instead. The
definition of Lipschitz differentiability then implies 
\[
f(x+h)=f(x)+\sum_{i=1}^{k}(Df(x))_{i}\langle h,x-x_{i}\rangle+o(\|h\|).
\]
 But then 
\[
\lim_{t\to0}\frac{f(x+th)-f(x)}{t}=\sum_{i=1}^{n}(Df(x))_{i}\langle h,x-x_{i}\rangle
\]
which means $f$ is directionally differentiable $\mathbf{m}$-almost
everywhere. 

Let $\mathbf{m}=f\mathcal{L}^{n}+\mathbf{m}^{s}$ be the Lebesgue
decomposition of $\mathbf{m}$. By \cite[Theorem 1.1(ii)]{Alberti2014}
there are Lipschitz functions which are not directionally differentiable
$\mathbf{m}^{s}$-almost everywhere, i.e. there is a Lipschitz function
$f$ such that for $\mathbf{m}^{s}$-almost everywhere $x\in\mathbb{R}^{n}$
there are $h_{x}\in\mathbb{R}^{n}$ such that 
\[
\liminf_{t\to0}\frac{f(x+th)-f(x)}{t}<\limsup_{t\to0}\frac{f(x+th)-f(x)}{t}
\]
but this can only happen if $\mathbf{m}^{s}=0$.
\end{proof}
The rigidity theorem can be strengthened as follows. 
\begin{prop}
[Rigidity Theorem II] If $(A,d_{\operatorname{Euclid}},\mathbf{m})$
with $A\subset\mathbb{R}^{n}$ is a differentiability space with Lipschitz
dimension equal to $n$ then $\mathbf{m}$ is absolutely continuous
w.r.t. the Lebesgue measure on $\mathbb{R}^{n}$. In particular, $\mathbf{m}$
is supported on an $n$-recitifiable subset. \end{prop}
\begin{rem*}
Without the Lipschitz dimension being $n$ the result is wrong. Examples
of constant dimension are given by Nash embeddings of Riemannian manifolds.
Being isometrically embedded in terms of Riemannian manifolds shows
that the extrinsic metric of $\mathbb{R}^{n}$ is (local) bi-Lipschitz.
An example with different dimensions is given by 
\[
X=\{(x,\mathbf{y})\in\mathbb{R}^{n}\,|\, x\le0,\|\mathbf{y}\|\le|x|\;\mbox{ or }\; x\ge0,y=0\}
\]
with metric induced by $\mathbb{R}^{n}$ and corresponding Lebesgue
measure on either side. Both examples are intrinsically differentiability
spaces, but as $\mathbf{m}\not\ll\mathcal{L}^{n}$ this is not true
extrinsically.\end{rem*}
\begin{proof}
Just decompose $\mathbf{m}$ into a sum of $k$-rectifiable measures
$\mathbf{m}^{k}$ and a purely unrectifiable measure $\mathbf{m}^{s}$
(see \cite[Section 2]{Alberti2014}). A $k$-rectifiable measure $\mathbf{m}^{k}$
has necessarily Lipschitz dimension equal to $k$. Note all but $\mathbf{m}^{0}$
are atomless. 

By \cite[Theorem 1.1(ii)]{Alberti2014} there is a Lipschitz function
$f$ such that at $\mathbf{m}^{s}$-almost every point $x\in M$,
$f$ is not directionally differentiable in any direction. However,
the existence of Alberti representation \cite[Theorem 7.8]{Bate2014}
implies that for $\mathbf{m}$-almost every $x\in M$ there is a direction
$h_{x}\in\mathbb{R}$ such that $f$ is directionally differentiable
at $x$ in direction $h_{x}$ \cite[Definition 2.11, Corollary 2.13]{Bate2014}.
But this means $\mathbf{m}^{s}=0$.
\end{proof}
Recall that subset $A\subset M$ is $n$-rectifiable if there is a
bi-Lipschitz map $f$ onto a measurable subset of $\mathbb{R}^{n}$.
Furthermore, a metric measure space is (countably) $\mathbf{m}$-rectifiable
if there exists a partition $\{B_{i}\}_{i\in\mathbb{N}}\cup\{B_{\infty}\}$
of $M$ with $\mathbf{m}(B_{\infty})=0$ and each $B_{i}$ is $n_{i}$-rectifiable
for some $n_{i}\in\mathbb{N}$. 
\begin{prop}
Assume $(M,d,\mathbf{m})$ is a differentiability space and $A_{i,n}$
is as above. If $A_{i,n}$ is $n$-recitifiable then $\mathbf{m}_{|A_{i,n}}$
absolutely continuous w.r.t. the $n$-dimensional Hausdorff measure
of $A_{i,n}$.\end{prop}
\begin{proof}
Let $\mu=f_{*}\mathbf{m}_{|A_{i,n}}$ where $f:A_{i,n}\to f(A_{i,n})\subset\mathbb{R}^{n}$
is bi-Lipschitz. Then $\mu$ has to be absolutely continuous w.r.t.
the Lebesgue measure on $\mathbb{R}^{n}$ or by bi-Lipschitzness equivalently
it has to be absolutely continuous w.r.t. $n$-dimensional Hausdorff
measure $(\mathcal{H}_{\tilde{d}}^{n})_{|f(A_{i,n})}$ induced by
push forward metric $\tilde{d}$ given by 
\[
\tilde{d}(x,y)=d(f^{-1}(x),f^{-1}(y))\qquad\mbox{ on }f(A_{i,n}).
\]
But then $\mathbf{m}_{|A_{i,n}}\ll\mathcal{H}_{d}^{n}$ where $\mathcal{H}_{d}^{n}$
is the $n$-dimensional Hausdorff measure on $(M,d)$. \end{proof}
\begin{cor}
Assume $(M,d,\mathbf{m})$ is Lipschitz differentiable such that each
$A_{i,n}$ is $n$-recitifiable. Then $(M,d,\mathbf{m})$ is $\mathbf{m}$-rectifiable
and $\mathbf{m}=\sum_{i,n}\varphi_{i,n}\mathcal{H}_{A_{i,n}}^{n}$
for some measurable functions $\varphi_{i,n}:M\to[0,\infty)$ with
$\varphi_{i,n}(x)=0$ for $x\notin A_{i,n}$. If $\mathbf{m}$ is
doubling then $\mathbf{m}=\sum_{i,n}\varphi_{i,n}\mathcal{H}^{n}$
where $\mathcal{H}^{n}$ is the Hausdorff measure on $(M,d)$. \end{cor}
\begin{rem*}
By \cite{Mondino2014} the above applies to $RCD(K,N)$-spaces. As
the theory of $RCD(K,N)$-spaces is a ``weighted'' theory it is
not true that $\varphi_{i,n}\equiv\mathsf{const}$, not even in the
smooth setting.
\end{rem*}
The result above shows that metric spaces equipped with a doubling
measure are in general not differentiability spaces and the Lipschitz
modules are not locally finite dimensional bounded. Furthermore, the
$L^{p}$-cotangent module might be locally finite dimensional without
$(M,d,\mathbf{m})$ being anywhere a differentiability space. A trivial
construction is as follows: Let $M$ be a differentiability space
with trivial $L^{p}$-cotangent module, i.e. $L_{\lip_{0}}^{p}(T^{*}M)=L_{\lip}^{p}(T^{*}M)$:
Let $([0,1],2^{-i}d_{\mbox{Euclid}},\mathbf{m}_{i})$ where $\mathbf{m}_{i}=2^{-i}w\boldsymbol{\lambda}$
where $w$ is the DiMarino-Speight $A_{p'}$-weights on $[0,1]$ with
$p<p'$ (see \cite{DiMarino2015}). Then the $L^{p}$-product of those
spaces is not a differentiability spaces in the sense above but its
$L^{p}$-structure is trivial. However, the product measure is not
doubling. As metrically doubling have reflexive Sobolev spaces \cite{Ambrosio2012}
we can ask the following. 
\begin{question*}
Is there a metric measure space $(M,d,\mathbf{m})$ with $\mathbf{m}$
being (locally) doubling that is not a differentiability space but
whose $L^{p}$-cotangent module is non-trivial and locally finite
dimensional bounded by some $N\in\mathbb{N}$? 
\end{question*}
The following is the central result of Cheeger's paper on a generalized
Rademacher Theorem (\cite{Cheeger1999}, see also \cite[Theorem 10]{Franchi1999}).
We state it in the context of the theory developed in this paper.
We refer the reader to the corresponding papers for exact definitions.
Compare the result also to \cite[Corollary 2.5.2]{Gigli2014}.
\begin{thm}
Assume $(M,d,\mathbf{m})$ is a metric measure space satisfying locally
uniformly doubling and $(1,p)$-Poincar\'e conditions. Then $(M,d,\mathbf{m})$
is a differentiability spaces, $d_{\lip}$ is closable, and the $L^{p}$-Lipschitz
cotangent and $L^{p}$-cotangent modules agree. \end{thm}
\begin{cor}
If, in addition, $(M,d)$ is bi-Lipschitz to \textup{$(\mathbb{R}^{n},d_{\operatorname{Euclid}})$}
then $\mathbf{m}$ is absolutely continuous w.r.t. the (pulled-back)
Lebesgue measure.\end{cor}
\begin{rem*}
This can be rephrased as saying if $(\mathbb{R}^{n},d_{\operatorname{Euclid}},\mathbf{m})$
satisfies the doubling and a $(1,p)$-Poinca\'e condition then $\mathbf{m}$
is absolutely continuous w.r.t. the Lebesgue measure. 
\end{rem*}

\subsection*{Cheeger differential structure and infinitesimally Hilbertian spaces}

On $\mbox{PI}{}_{p}$-spaces, i.e. those satisfying a doubling and
$(1,p)$-Poincar\'e condition, one frequently ``renorms'' the $L^{p}$-cotangent
module to obtain a (pointwise) inner product $(f,g)\mapsto Df\cdot Dg$
with $Df$ the Cheeger differential of $f$. This structure is then
called Cheeger differential structure. The choice of scalar product
is, however, highly non-unique. Using Theorem \ref{thm:reflexive-renorm}
one may choose the John or Binet-Legendre scalar product given by
Theorem \ref{thm:John-scalar} to obtain unique scalar product so
that $Df$ only depends on the charts and basis elements. 

For a certain class of spaces the renorming is actually trivially
and thus not necessary: Recall that a space is called \emph{infinitesimally
Hilbertian} (see \cite{AGS2011,Gigli2012}) if $E\mathcal{F}_{2}$
is a quadratic form and thus a strongly local, closed and Markovian
Dirichlet form. Then one can show that it is equivalent to $f\mapsto|Df|_{*,2}^{2}$
being a quadratic form $\mathbf{m}$-almost everywhere. But then the
$L^{2}$-cotangent module is a Hilbert module and therefore reflexive.
More generally we may say that the space is \emph{infinitesimally
Riemannian (w.r.t. $\mathbf{m}$) }if $f\mapsto(\lip f)^{2}$ is quadratic
$\mathbf{m}$-almost everywhere. 
\begin{rem*}
The terminology infinitesimally Hilbertian stems from the fact the
Sobolev space $W^{1,2}(M,\mathbf{m})$ is a Hilbert space iff the
space is infinitesimally Hilbertian. We choose infinitesimally Riemannian
as it only relates to Lipschitz functions and thus to the metric $d$
itself. This class includes sub-Riemannian manifolds, like the class
of Heisenberg groups with Carnot-Caratheodory metric induced by an
inner product on the horizontal bundle.

The characterization of the Lipschitz and cotangent modules via Theorem
\ref{thm:zero-Lipschitz-module} also shows.\end{rem*}
\begin{prop}
Assume $(M,d,\mathbf{m})$ is infinitesimally Riemannian and $\mathbf{m}$
is finite on bounded sets. Then $(M,d,\mathbf{m})$ is infinitesimally
Hilbertian and $L^{p}(T^{*}M)$ is reflexive for $p\in(1,\infty)$.\end{prop}
\begin{rem*}
One can also show $p$-infinitesimal strict convexity (see \cite{Gigli2012})
as the dual of $L^{p}(T^{*}M)$ also has pointwise norm induced a
scalar product which makes its (pointwise) norm strictly convex. \end{rem*}
\begin{proof}
If $x\mapsto|\cdot|^{2}$ is quadratic then $x\mapsto|\cdot|$ is
$2$-uniformly convex and $2$-uniformly smooth. Thus $x\mapsto|\cdot|$
is $p$-uniformly convex for $p\ge2$ and $p$-uniformly smooth if
$p\le2$ with convexity and resp. smoothness constants only depending
on $p$. But then 
\[
f\mapsto\|\chi_{A}d_{\lip}f\|_{L^{p}(T^{*}M)}
\]
is $p$-uniformly convex/smooth with the same constant. In particular,
$L_{\lip}^{p}(T^{*}M)$ is (super)reflexive. As $L^{p}(T^{*}M)$ is
a quotient space of a reflexive space it is itself reflexive. 

In case $p=2$ one sees that 
\[
v\mapsto\|v\|_{L_{\lip}^{2}(T^{*}M)}^{2}=\int|v|^{2}d\mathbf{m}
\]
 is quadratic and thus $L_{\lip}^{2}(T^{*}M)$ a Hilbert module. But
in Hilbert spaces quotient spaces are isometric to orthogonal subspaces,
i.e. 
\[
L^{2}(T^{*}M)\cong\left(L_{\lip_{0}}^{2}(T^{*}M)\right)^{\perp},
\]
showing that $L^{2}(T^{*}M)$ is a Hilbert module and $E\mathcal{F}_{2}$
quadratic.
\end{proof}
If it is a priori known that the $\mbox{Lipschitz}_{0}$-module is
trivial then the converse is true as well.
\begin{cor}
Assume $\mathbf{m}$ is bounded on finite subsets and $L_{\lip_{0}}^{2}(T^{*}M)$
is trivial. Then $(M,d,\mathbf{m})$ is infinitesimally Riemannian
iff it is infinitesimally Hilbertian. In either case the Sobolev spaces
are reflexive. \end{cor}
\begin{example*}
(1) Every infinitesimally Hilbertian $\mbox{PI}_{p}$-space with $p\le2$
is infinitesimally Riemannian and its Cheeger differential structure
gives a representation of the Sobolev differentials. If $p>2$ the
$L^{2}$-cotangent module might be trivial and thus trivially infinitesimally
Hilbertian.

(2) By \cite{Gigli2014a}, on $RCD(K,\infty)$-space the relaxed slope/weak
upper gradient is independent of $p$ and $|Df|_{*,p}=\lip f$ for
all $f\in D(\mathcal{F}_{p}^{\lip})$. For those spaces the $\mbox{Lipschitz}_{0}$-modules
are trivial but $L^{2}$-cotangent module might not be locally finite
dimensional. Note that this result only needs the Bakry-\'Emery condition
and is independent of the proof of a $(1,1)$-Poincar\'e condition
on $RCD(K,N)$-spaces.
\end{example*}

\subsection*{Further results}

To finish this section let's see what happens in the finite measure
case. For this assume for simplicity that $\mathbf{m}$ is a probability
measure. Using cut-off functions and the chain rule one can also show
the results below for $\sigma$-finite measures.
\begin{lem}
Assume $\mathbf{m}$ is finite. Then non-triviality of $L_{\lip_{0}}^{p}(T^{*}M)$
implies that $L_{\lip_{0}}^{p'}(T^{*}M)$ is non-trivial for all $1<p'\le p$.
Or equivalently, closability of $d_{\lip}$ in $L_{\lip}^{p'}(T^{*}M)$
implies closability of $d_{\lip}$ in $L_{\lip}^{p}(T^{*}M)$ for
all $p'\le p<\infty$. \end{lem}
\begin{proof}
First note that $D(\mathcal{F}_{p}^{\lip})=\Lip(M,d)$ for $p\in(1,\infty)$.
From Hölder inequality we get 
\begin{eqnarray*}
\|d_{\lip}f\|_{L_{\lip}^{p'}(T^{*}M)} & = & \left(\int(\lip f)^{p'}d\mathbf{m})\right)^{\frac{1}{p'}}\\
 & \le & \mbox{\ensuremath{\mathbf{m}}}(M)\left(\int(\lip f)^{p}d\mathbf{m})\right)^{\frac{1}{p}}=\mbox{\ensuremath{\mathbf{m}}}(M)\|d_{\lip}f\|_{L_{\lip}^{p}(T^{*}M)}.
\end{eqnarray*}
So if $(d_{\lip}f_{n})_{n\in\mathbb{N}}$ is Cauchy in $L_{\lip}^{p}(T^{*}M)$
then it is also Cauchy in $L_{\lip}^{p'}(T^{*}M)$. Thus $L_{\lip}^{p}(T^{*}M)\subset L_{\lip}^{p'}(T^{*}M)$
as the spaces are defined via completion, i.e. ``space of Cauchy
sequences''. So assume $d_{\lip}f_{n}\to\omega\in L_{\lip}^{p}(T^{*}M)$
then also $d_{\lip}f_{n}\to\omega\in L_{\lip}^{p}(T^{*}M)$. But if
$|\omega|\ne0$ on a set of positive $\mathbf{m}$-measure then $\omega\ne\mathbf{0}\in L_{\lip}^{p'}(T^{*}M)$.
In particular, if $d_{\lip}$ is closable in $L_{\lip}^{p'}(T^{*}M)$
then so in $L_{\lip}^{p}(T^{*}M)$. 
\end{proof}
The result is a generalized version of what happens if a $(1,p')$-Poincar\'e
inequality holds: In such a case, one can show that $|Df|_{*,p'}=\lip f$
so that $d_{\lip}$ is closable in all $L_{\lip}^{p}(T^{*}M)$ with
$p\in[p',\infty)$. Furthermore, it is known there is a $p_{\infty}\in(1,\infty)$
such that for $p>p_{\infty}$ the $(1,p)$-Poincar\'e condition holds
but not the $(1,p_{\infty})$-Poincar\'e condition (see \cite{Keith2008}).
In this general setting we cannot prove this fact yet.
\begin{question}
[Open ended condition] Assume $\mathbf{m}$ is finite and $L_{\lip}^{p}(T^{*}M)$
weakly reflexive for all $p\in(1,\infty)$. Is $L_{\lip_{0}}^{p_{\infty}}(T^{*}M)$
is non-trivial for the $p_{\infty}$ given above? 
\end{question}
Even more general one may wonder whether the condition here is strictly
weaker even on spaces with Poincar\'e condition.
\begin{question}
Is there a metric space with a doubling measure satisfying a $(1,p)$-Poincar\'e
condition such that for some $p'\in(1,p)$ the module $L_{\lip_{0}}^{p'}(T^{*}M)$
is trivial but the $(1,p')$-Poincar\'e condition does not hold?
\end{question}
\appendix

\section{Selection of scalar products}

Let $\Norm(\mathbb{R}^{n})$ be the space of norms on $\mathbb{R}^{n}$.
We can equip this space with the following intrinsic metric
\[
d(F,F')=\sup_{v\in\mathbb{R}^{n}\backslash\{0\}}\{\ln\frac{F(v)}{F'(v)},\ln\frac{F'(v)}{F(v)}\}.
\]
Indeed, $d(F,F')=d(F,F')\ge0$ with equality iff $F(v)=F'(v)$ for
all $v\in\mathbb{R}^{n}$. Choose $v\in\mathbb{R}^{n}\backslash\{0\}$
then 
\[
\ln\frac{F(v)}{F'(v)}\le\ln\left(\frac{F(v)}{F"(v)}\frac{F"(v)}{F'(v)}\right)\le\ln\frac{F(v)}{F"(v)}+\ln\frac{F"(v)}{F'(v)}
\]
which immediately yields the triangle inequality. The interested reader
may verify that $(\Norm(\mathbb{R}^{n}),d)$ is a proper metric space
and its topology agrees with the $C^{0}$-topology. In particular,
the space is separable.
\begin{rem*}
An equivalent characterization is via the constant of uniform equivalent:
$d(F,F')$ is the infimum over all $C>0$ such that $C^{-1}F\le F'\le FC$. 
\end{rem*}
Every scalar product induces a norm on $\mathbb{\mathbb{R}}^{n}$
that satisfies pointwise the parallelogram equations. As $d$-convergence
also implies pointwise convergence, the space of scalar products $\Scalar(\mathbb{R}^{n})$
can be seen as a closed subspace in $\Norm(\mathbb{R}^{n})$. 
\begin{thm}
[John ellipsoid] \label{thm:John-scalar}For any $F\in\Norm(\mathbb{R}^{n})$
there is a unique scalar product $g_{J}^{F}$ such that for all $v\in\mathbb{R}^{n}$
\[
g_{J}^{F}(v,v)\le F^{2}(v)\le ng_{J}^{F}(v,v).
\]
In particular, $F\mapsto g_{J}^{F}$ is continuous from the space
of norm $\Norm(\mathbb{R}^{n})$ to the space of scalar products $\Scalar(\mathbb{R}^{n})$.\end{thm}
\begin{proof}
Existence is well-known (see). For continuity just observe that $(\Norm(\mathbb{R}^{n}),d)$
is proper and by the characterization $d(F,g_{J}^{F})\le\sqrt{n}$.
Continuity then follows from uniqueness of the John ellipsoid. 
\end{proof}
Alternatively one may use the Binet ellipsoid $g_{BL}^{F}$of the
unit sphere w.r.t. $F$ which induces a scalar product called the
\emph{Binet-Legendre scalar product}. The assignment satisfies $d(g_{BL}^{F},g_{BL}^{F'})\le nd(F,F')$
and $d(F,g_{BL})\le\sqrt{n}(n+2)$. We refer to \cite{Matveev2012}
for more details on this construction and its use in Finsler geometry. 

\bibliographystyle{amsalpha}
\bibliography{bib}

\end{document}